\documentclass[12pt,reqno,tbtags]{amsart}
\usepackage{amssymb}
\numberwithin{equation}{section}

\usepackage[usenames]{color}

\newcommand{\ch}[1]{#1}
\newcommand{\chs}[1]{#1}
\newcommand{\chfin}[1]{#1}

\newcommand{\halmos}{\rule{1ex}{1.4ex}}
\newcommand{\proofbox}{\hspace*{\fill}\mbox{$\halmos$}}
\newenvironment{proofof}[1]{\noindent {\it
Proof of #1}.}{\proofbox\par\smallskip\par}
\newenvironment{prf}{\noindent {\it
Proof}.}{\proofbox\par\smallskip\par}

\newtheorem{theorem}{Theorem}[section]
\newtheorem{lemma}[theorem]{Lemma}
\newtheorem{proposition}[theorem]{Proposition}
\newtheorem{corollary}[theorem]{Corollary}

\theoremstyle{definition}

\theoremstyle{remark}

\newcounter{thmenumerate}

\newcounter{xenumerate}

\newcommand\REM[1]{\texttt{[#1]}\marginal{XXX}}

\begingroup
  \divide\count255 by 60
  \multiply\count255 by -60
  \advance\count255 by \time
  \ifnum \count255 < 10 \xdef\klockan{\the\count1.0\the\count255}
  \else\xdef\klockan{\the\count1.\the\count255}\fi
\endgroup

\def\rompar(#1){\textup(#1\textup)}    

\def\xexp(#1){e^{#1}}

\newcommand\ntoo{\ensuremath{{n\to\infty}}}

\newcommand\iid{i.i.d.\spacefactor=1000}

\newcommand\eg{e.g.\spacefactor=1000}

\newcommand\tN{\ensuremath{\tilde{N}}}
\newcommand\tvep{\ensuremath{\tilde{\vep}}}

\newcommand\Op{O_{\mathrm p}}
\newcommand\op{o_{\mathrm p}}

\newcommand\bbN{\mathbb N}

\newcommand\E{\operatorname{\mathbb E{}}}
\renewcommand\P{\operatorname{\mathbb P{}}}

\newcommand\Bi{\operatorname{Bi}}

\newcommand\gd{\delta}

\newcommand\gl{\lambda}

\newcommand\gs{\sigma}

\newcommand\ep{\varepsilon}
\renewcommand\phi{\varphi}

\newcommand\cB{\mathcal B}

\def\[#1]{[\![#1]\!]}

\renewcommand{\=}{:=}

\newcommand{\mgf}{moment generating function}

\newcommand\rv{random variable}

\newcommand\Cv{\ensuremath{C({\bf v})}}
\newcommand\cv{\ensuremath{|C({\bf v})|}}

\newcommand{\vv}{\ensuremath {\bf v}}

\newcommand\whp{\textbf{whp}}
\newcommand\wvhp{\textbf{wvhp}}

\newcommand\marginal[1]{\marginpar{\raggedright\parindent=0pt\tiny #1}}

\newcommand\REMRM[1]{}

\newcommand\Vto{\ensuremath{{V \to\infty}}}
\newcommand{\gr}{{\mathbb G}}
\newcommand{\eq}{\begin{equation}}
\newcommand{\en}{\end{equation}}
\newcommand{\lbeq}[1]{\label{#1}}
\newcommand{\refeq}[1]{(\ref{#1})}
\newcommand{\Cmax}{{\mathcal C}_{\rm max}}
\newcommand{\eqn}[1]{\eq #1 \en}
\newcommand{\cn}{\Omega}
\newcommand{\vep}{\varepsilon}
\newcommand{\prob}{{\mathbb P}}
\newcommand{\expec}{{\mathbb E}}
\newcommand{\R}{{\mathbb R}}
\newcommand{\var}{{\rm Var}}
\newcommand{\sss}{\scriptscriptstyle}

\newcommand{\sT}{{\sss T}}

\newcommand{\eqan}[1]
{\begin{align}
#1
\end{align}
}
\newcommand{\conn}{\longleftrightarrow}
\newcommand{\conns}{\leftrightarrow}
\newcommand{\nc}  { \conn  {\hspace{-2.7ex} /} \hspace{1.8ex}   }
\newcommand{\ncs}  { \conns  {\hspace{-1.0ex} /} \hspace{0.3ex}   }
\renewcommand{\Pr}{{\mathbb P}_p}
\newcommand{\Zbar}{{\bar Z}}

\newcommand{\NV}{\overline{N}}
\newcommand{\Nul}{\underline{N}}
\newcommand{\tcn}{\ensuremath{\tilde \cn}}
\newcommand{\vepnull}{\varepsilon_0}

\begin{document}
\title{Random subgraphs of the 2D Hamming graph:
\linebreak The supercritical phase}

\date{27 December 2007 (\today)}

\author{Remco van der Hofstad}
\address{Department of Mathematics and
        Computer Science, Eindhoven University of Technology,
        5600 MB Eindhoven, The Netherlands.}
\email{rhofstad@win.tue.nl}
\urladdr{http://www.win.tue.nl/$\sim$rhofstad}

\author{Malwina J. Luczak}
\address{Department of Mathematics, London School of Economics,
  Houghton Street, London WC2A 2AE, United Kingdom}
\email{m.j.luczak@lse.ac.uk}
\urladdr{http://www.lse.ac.uk/people/m.j.luczak@lse.ac.uk/}


\begin{abstract}
We study random subgraphs of the 2-dimensional Hamming graph $H(2,n)$,
which is the Cartesian product of two complete graphs on $n$
vertices. Let $p$ be the edge probability, and write
$p=\frac{1+\vep}{2(n-1)}$ for some $\vep\in \R$.
In \cite{bchss1, bchss2}, the size of the largest connected component
was estimated precisely for a large class of graphs including $H(2,n)$
for $\vep\leq \Lambda V^{-1/3}$, where $\Lambda > 0$ is a
constant and $V=n^2$ denotes the number of vertices in
$H(2,n)$. Until now, no matching lower bound on the size in the
supercritical regime has been obtained.

In this paper we prove that, when $\vep\gg (\log{V})^{1/3} V^{-1/3}$,
then the largest connected component has size close to $2\vep V$
with high
probability. We thus obtain a law of large numbers
for the largest connected component size, and show
that the corresponding values of $p$ are supercritical. Barring the
factor $(\log{\chs{V}})^{1/3}$, this identifies the size of the largest connected
component all the way down to the critical $p$ window.
\end{abstract}

\maketitle

\section{Introduction}\label{S:intro}

We study random subgraphs of the 2-dimensional Hamming graph $H(2,n)$.
The $d$-dimensional Hamming graph is a
graph on $V=n^d$ vertices, each corresponding to one of the
$n^d$ distinct $d$-vectors ${\bf v} =(v_1, \ldots , v_d) \in \{1, \ldots,
n\}^d$. A pair of vertices are connected by an edge if and only if these
vertices differ in precisely one \chs{coordinate}. (See for example~\cite{cs}
for more information on the properties of Hamming graphs.) The 1-dimensional
Hamming graph $H(1,n)$ is the complete graph; for $d\geq 2$, the
graph
$H(d,n)$ is the Cartesian product of $d$ complete graphs on $n$
vertices. In particular, it is transitive and the degree of each
vertex is  $\cn=d(n-1)$.

We write $\prob_p$ for the probability law of the
random subgraph of $\gr$ resulting when each edge is {\em occupied} (or {\em
present}) with
probability $p$, and {\em vacant} (or {\em absent}) with probability
$1-p$, independently of all the other edges. We write $\E_p$ for the
expectation with respect to $\prob_p$. Also, $\var_p$ will denote the
variance under $\prob_p$.

Throughout we work with the 2-dimensional Hamming graph $H(2,n)$
unless explicitly stated otherwise, and we shall assume
that $p=\frac{1+\vep}{2(n-1)}= \frac{1+ \vep} {\cn}$, where
$\vep = \vep (n) \in (0,1)$ tends to 0 in a certain way to be
specified below. Our goal is to study properties of random subgraphs
of $H(2,n)$ under $\prob_p$.

Random subgraphs of finite tori with various edge sets were studied in
quite some generality in \cite{bchss1, bchss2}, and we now highlight the
key results of these papers. Some of the theorems in \cite{bchss1, bchss2}
apply to a general finite
transitive graph, which in what follows will be denoted by $\gr$. We
also denote the number of vertices \ch{or volume of $\gr$ by $V=|\gr|$}
and the vertex degree by $\Omega$. Given a vertex ${\bf
v}$ of $\gr$, we shall write $\Cv$
for the {\it connected component} or {\em cluster} containing ${\bf v}$,
and $\cv$ for the number of vertices in $\Cv$.
Further, we let $\chi(p)$ be the expected size of the cluster
containing ${\bf v}$, that is
    \eqn{
    \lbeq{chi-def}
    \chi(p)=\expec_p[|\Cv|].
    }
(Note that, by transitivity, this is independent of the
choice of ${\bf v}$.) Then in \cite{bchss1, bchss2}
the {\em critical threshold}
$p_c=p_c(\gr,\lambda)$ of a finite transitive graph $\gr$
is defined to be the unique solution to the equation
    \eq
    \lbeq{pcdef}
    \chi(p_c) =  \lambda V^{1/3},\,
    \en
where $\lambda >0$ is a sufficiently small
constant. (See~\cite{bchss2} for details concerning the precise
constraints on the size of $\lambda$.)

In \cite{bchss1}, cluster sizes were investigated for graphs $\gr$
satisfying the so-called \chs{triangle condition.} In \cite{bchss2}, the
triangle condition was established for certain types of graphs
$\gr$, including the Hamming graph $H(d,n)$ of a general dimension
$d \ge 1$. We shall now describe these results briefly in order to
set up our own scene.

Let $\Cmax$ denote a cluster of maximum size, where
we may pick any such cluster if it is not unique. Then
$|\Cmax|$ is the maximum cluster size, that is
    \eq
    \lbeq{Cmax}
    |\Cmax|=\max\{|C({\bf v})|: {\bf v}\in \gr\}.
    \en
The main theorems in \cite{bchss1} concern the scaling of
$\chi(p)$ and bounds on $|\Cmax|$ in graphs \ch{$\gr$} satisfying
the triangle condition as \ch{$|\gr|=\Vto$.} Specifically, it
is shown in \cite{bchss1, bchss2} that, if $p_c$ is as in
\refeq{pcdef} and
    \eqn{
    \lbeq{p-def}
    p=p_c+\frac{\vep}{\cn},
    }
then, for all $\vep$ such that $\vep V^{1/3}\rightarrow -\infty$,
asymptotically the expected cluster size $\chi(p)$ satisfies
    \eqn{
    \lbeq{chiasysub}
    \chi(p)=\frac{1+O(\cn^{-1})+O(V^{-1/3})}{|\vep|}.
    }
With regard to the maximum cluster size, for all $\omega\geq 1$, as
$V \to \infty$,
    \eqn{
    \lbeq{scalingsub}
    \prob_p\Big(\frac{\chi^2(p)}{3600\omega}\leq |\Cmax|\leq 2\chi^2(p)\log(V/\chi^3(p))
    \Big)\geq \big(1+\frac {36\chi^3(p)}{\omega V}\Big)^{-1}-{\sqrt{e}}{[2\log(V/\chi^3(p))]^{-3/2}}.
    }
The above describes the behaviour of the mean and
maximum cluster sizes for {\it subcritical}
$p$ values, which are $p$ values satisfying $\vep V^{1/3}\rightarrow -\infty$;
in particular, the bounds apply to $H(2,n)$.

For a constant $\Lambda > 0$, the {\it critical
window} is defined as the interval of all $p=p_c+\frac{\vep}{\cn}$
such that $|\vep|\leq \Lambda V^{-1/3}$. Theorem 1.3 in~\cite{bchss1} shows that,
for some constant $b=b(\Lambda)$, the maximum cluster size inside the
critical window satisfies
    \eqn{
    \lbeq{LCBd1-win}
    \prob_p\Big(
    \omega^{-1} V^{2/3}\leq |\Cmax|\leq \omega V^{2/3}
        \Big)
    \geq 1-\frac{b}\omega.}

The corresponding results in \cite{bchss1, bchss2} are significantly weaker
in the case $p=p_c+\frac{\vep}{\cn}$ where $\vep^3V\rightarrow
\infty$ (that is, when $p$ is {\em above the critical window}
or {\it supercritical}). In particular, only upper bounds on the
maximum cluster size are established therein. More precisely, it is proved
in~\cite{bchss1} that, for all $\omega\geq 1$,
    \eqn{
    \prob_p
    \Big(|\Cmax|\geq \omega (V^{2/3}+\vep V)\Big) \leq \frac{21}\omega.
    }
The problem with this result is that it does not imply that $p_c$
as defined in~(\ref{pcdef}) actually {\it is} the critical value, and thus that
$p=p_c+\frac{\vep}{\cn}$ with $\vep^3V\rightarrow
\infty$ really is above the critical window. Indeed, to prove that
this is the case, one additionally needs a {\it
lower bound} on the maximum connected component size.
No such results are established in \cite{bchss1, bchss2}, and we expect
that the geometry of the graphs under consideration plays a crucial
role in lower bounding the largest cluster size.
\smallskip

The aim of this paper is to establish the asymptotics of the maximum
supercritical cluster for the 2-dimensional Hamming graph $H(2,n)$.
Throughout our proofs we shall use the phrase ``with high
probability'' (abbreviated as ``\whp'') to mean ``with probability
tending to 1 as $\Vto$''. Also, ``with very high probability''
(abbreviated as ``\wvhp'') will mean ``with probability at least $1
- O(V^{-3})$ as $\Vto$''. All unspecified limits are as $\Vto$.
Given an event $E$, $I[E]$ will denote the indicator of $E$. We
write $\prob (\cdot )$ for a generic probability measure (for
instance, the probability measure corresponding to a sequence of
\iid{} binomial random variables), which may vary from situation to
situation. We use the $\Op$ and $\op$ notations in the standard way
(see \eg{} Janson, {\L}uczak and Ruci\'nski~\cite{JLR}). For
example, if $(X_n)$ is a sequence of random variables, then
$X_n=\Op(1)$ means ``$X_n$ is bounded in probability'' and
$X_n=\op(1)$ means that $X_n$ converges to zero in probability as
$\ntoo$. We shall also use the asymptotic $o(), O(),\Omega (),
\Theta()$ notations (without the subscript ``${\rm p}$'') in the
standard way, and again referring to the regime where $\Vto$. We
write $f(V)\gg g(V)$ (resp.\ $f(V)\ll g(V)$) when $g(V)=o(f(V))$
(resp.\ $f(V)=o(g(V))$) as $\Vto$. \ch{Throughout, the symbol ``$\sim$''
refers to, often heuristic, estimates of the {\em leading order}
as $\Vto$, with unspecified constants and thus uncontrolled error terms.}
Finally, we denote by $C$ a generic (unspecified)
positive constant, which may change from line to line. We shall
interchange this use of $C$ with the $O()$ notation.

\subsection{The model}
\label{sec-model}
We consider the Hamming graph $H(d,n)$, and
take the edge probability \chs{$p=(1+\vep)/\cn$.
We first argue that this agrees asymptotically with the choice
of $p$ in \refeq{p-def}.}
Let us note that~\cite[Theorem 1.5]{bchss1}
establishes that, for a graph $\gr$ satisfying the
triangle condition,
    \eqn{
    \lbeq{pcchires}
    1-\chi(p_c)^{-1}
    \leq \cn p_c \leq
    1-\chi(p_c)^{-1}+O(\cn^{-1}).
    }
When $\gr = H(d,n)$, then
$\cn=d(n-1)$ and $\chi(p_c)=\lambda V^{1/3}
=\lambda n^{d/3}$. Therefore, if $\vep = \Theta (V^{-1/3})$, then
    \eqn{
    \lbeq{critcomp1}
    p=\frac{1+\vep}{\cn } = p_c + \frac{\vep}{\cn}(1+O(1)),
    }
while for $p$ outside the critical window,
    \eqn{
    \lbeq{critcomp2}
    p=\frac{1+\vep}{\cn } = p_c + \frac{\vep}{\cn}(1+o(1)).
    }
Since in the case $d=2$, \ch{we have that }
$\cn^{-2}= o(\chs{\vep/\cn})$ for $\vep\gg
V^{-1/3}$, the critical value defined in \cite{bchss1, bchss2}
agrees asymptotically \ch{to leading order}
with the value \chs{$1/d(n-1)= 1/\cn$}.
In particular, $p=1/d(n-1)$ is inside the critical window of
\cite{bchss1, bchss2}. This shows that we are working in the correct range of
$p$ values.
For $d\geq 3$, \refeq{critcomp1}--\refeq{critcomp2} may not
necessarily be valid, and we shall
discuss this issue in more detail in Section \ref{sec-DiscHeur}.

From now on, we concentrate on the supercritical case,
that is $\vep\gg V^{-1/3}=n^{-d/3}$. Our main result is
the following:

\begin{theorem}
[The supercritical phase for $H(2,n)$]
\label{thm-maind2}
Consider the 2-dimensional Hamming graph
\linebreak
$H(2,n)$. Let $p=p_c +\frac{\vep}{\cn}$ and let
$\chs{V^{-1/3}(\log{V})^{1/3}}\ll
\vep \ll 1$.
Then
    \eqn{
    |\Cmax|=2\vep n^2(1+\op(1)).
    }
\end{theorem}

\smallskip

\noindent
Theorem \ref{thm-maind2} shows that, when $n^{-2/3}(\log{n})^{1/3} \ll
\vep \ll 1$, the largest connected component satisfies a law of large numbers.
Barring the factor $(\log{\chs{V}})^{1/3}$ in the lower bound
on $\vep$, Theorem \ref{thm-maind2} identifies the asymptotic size
of the largest cluster all the way down to the
critical threshold. Therefore, our result
demonstrates that $p_c=\frac{1}{2(n-1)}$ really {\it is}
the critical value for random subgraphs of the 2-dimensional
Hamming graph. We believe that our proof can be adapted to deal with
the case where $\vep>0$ is {\it fixed}. Here, the corresponding
statement would be that
$|\Cmax|\sim \zeta_{1+\vep} V$ \whp, where $\zeta_\lambda$ is
the survival probability of a Poisson branching process with
mean offspring $\lambda$. Since the proof of
Theorem \ref{thm-maind2} is the most challenging for $\vep$
as close as possible to the critical window $V^{-1/3}$,
we choose not to consider the case of constant $\vep$ in this paper.

\ch{Before} giving a proof of Theorem \ref{thm-maind2}, we
discuss its statement in more detail in Section \ref{sec-DiscHeur}
below. Therein we also include some conjectures concerning Hamming
graphs of a general dimension $d$.

\subsection{Discussion and heuristics}
\label{sec-DiscHeur}
We first sketch an intuitive picture justifying the definition of
$p_c$ from \cite{bchss1, bchss2} given in~\refeq{pcdef}.
This picture relies on a branching process
approximation for $p<p_c$.

We expect random clusters in our model to exhibit behaviour similar
to that of a subcritical branching process. Therefore, from the
theory of branching processes, if $p=p_c + \frac{\vep}{\cn}$ is just
below the critical point (for instance, if $\vep < 0$), then we
should have \chs{(e.g.,\ from the Otter-Dwass formula, see Lemma
\ref{lem-OD} below)}
    \eqn{
    \prob_p(|C({\bf v})|\geq k)\sim \frac{1}{\sqrt{k}}
    e^{-\frac12 k \vep^2(1+o(1))},
    }
which in turn implies that
    \eqn{
    \chi(p)=\expec_p[|C({\bf v})|] \sim \int_0^{\infty} x^{-1/2}
    e^{-\frac12 x \vep^2(1+o(1))} dx
    \sim \int_0^{\vep^{-2}} x^{-1/2} dx
    \sim \ep^{-1}.
    }
Thus, in fact,
    \eqn{
    \lbeq{subprobkasy}
    \prob_p(|C({\bf v})|\geq k) \sim
    \frac{1}{\sqrt{k}} e^{-\frac{k
    }{\chi^2(p)}\Omega(1)},}
and hence for subcritical $p$ (possibly up to logarithmic corrections)
    \eqn{
    \lbeq{psubcrit}
    |\Cmax|\sim \chi(p)^2 \quad \whp.
    }

On the other hand, \chs{in the case $p>p_c$} there should be a
connected component dominating all the others in size. One way to
express this intuitive statement is to impose that
    \eqn{
    \lbeq{psupcrit}
    \chi(p)=\expec_p[|C({\bf v})|]\sim
    \expec_p\big[|C({\bf v})|I[{\bf v}
    \in \Cmax]\big]   =\frac{1}{V}\expec_p[|\Cmax|^2].
    }
Naturally, the meaning of formula~(\ref{psupcrit}) is, in essence, that
the main contribution to the expected size of a cluster of any
particular vertex $\vv$ is from those configurations where this vertex
lies in the largest component.

Note that~(\ref{psupcrit}) could be taken as a {\it defining property}
of supercritical behaviour. Then the {\it critical window} can be
defined as the interval of $p$ values where
the subcritical and supercritical pictures coincide. In other words, if $p$ lies within
the critical window then both~\refeq{psubcrit} and~\refeq{psupcrit}
should be satisfied.

Assume further that a sufficient amount of the concentration of
measure exhibited by $|\Cmax|$ in the subcritical regime (as implied
by \refeq{psubcrit}) carries through to the critical window, so that
$\expec_p[|\Cmax|^2] \sim \chi (p)^4$. It then follows that, for $p$
inside the critical window,
    \eqn{
    \chi(p)\sim \frac{1}{V}\expec_p[|\Cmax|^2]
    \sim \frac{1}{V} \chi(p)^4;
    }
and hence, inside the critical window, \ch{we are led to}
    \eqn{
    \chi(p)\sim V^{1/3}.
    }
This provides a rationale for the definition~\refeq{pcdef} of the critical
threshold $p_c$. In conclusion, the above heuristic demonstrates that
branching process approximations in the subcritical \chs{regime} and
the domination of the expected cluster size
by the maximum cluster size
in the supercritical \chs{regime} together imply that \refeq{pcdef}
is the ``right'' definition for $p_c$.

At this point, we emphasise that subcritical branching process
approximations are only likely to be valid
for a random graph \ch{that is} sufficiently mean-field in character,
in the sense that its geometry is of little significance for
the structure of its random subgraphs.
This is the case for sufficiently high-dimensional random
graphs, but cannot be expected to hold for
low-dimensional random graphs, as indicated in~\cite{BCKS99, BCKS01}.
For random subgraphs of the
torus with nearest-neighbour bonds in a sufficiently high (but
constant) dimension, as well as for the torus with sufficiently
spread-out bonds in \chs{dimensions} greater than 6, it is shown in
\cite{HeHof05} that the largest critical connected component
is of order $V^{2/3}$, with logarithmic corrections in
the lower bound. Accordingly, assuming universality in
high-dimensional finite-range percolation,
one can expect classical random graph
asymptotics at criticality to be valid for random subgraphs
of the torus when $d>6$ for general choices of finite-range edges.
On the other hand, the results of
\cite{BCKS99, BCKS01} suggest that random graph asymptotics at the
phase transition threshold are {\it not} valid for random subgraphs of
the $d$-dimensional torus when $d<6$.

We close this section with a few comments and conjectures. The
present paper verifies the location of the critical window found
in \chs{\cite{bchss1, bchss2}} up to a factor $(\log \chs{V})^{1/3}$. The
main barrier to overcoming this separation with our approach is the
fact that we require concentration of measure for the number of
vertices with either the first or second \chs{coordinate} fixed in order
for our estimates to be sufficiently precise; this concentration
property fails when $\vep$ is too small. Similar issues cause
problems with extensions of our approach to $H(d,n)$ for $d > 2$,
although we believe that it could handle the $d=3$ case. Let us
mention at this point that the $(\log \chs{V})^{1/3}$ separation has since
been removed by Nachmias~\cite{n07}, using non-backtracking random
walks. He also manages to nail down the critical window from above
in $H(3,n)$, although he does not establish laws of large numbers
for the giant \chs{component for $H(2,n)$ or $H(3,n)$.}

In the present paper we investigate the scaling of the largest
connected component in supercritical percolation on the Hamming
graph $H(2,n)$. Many random graph models are well known to satisfy
what is sometimes referred to as a {\it discrete duality principle}
(see for instance~\cite[Section 10.5]{AS00}). This is the principle
that the size of the second largest supercritical component is
asymptotically close in distribution to the size of the largest
subcritical component for an appropriate choice of subcritical edge
probability. This notion of duality is closely related to the
duality exhibited by branching
processes~\cite{an72,devroye,dwass,kolchin}, or \cite[Section 10.4]
{AS00}. We expect the Hamming graph $H(2,n)$ to follow the
discrete duality principle. More precisely, we expect that if we
were to remove the largest connected component when $p=p_c+\vep/\cn$
with $\vep\gg V^{-1/3}$, then the resulting connected components
would be like those of the Hamming graph with $p=p_c-\vep/\cn$. In
particular, letting $|{\mathcal C}_{\sss (2)}|$ be the size of the
second largest component, we conjecture that
    \eqn{
    \lbeq{dualprin}
    |{\mathcal C}_{\sss (2)}|= 2\vep^{-2}\log(\vep^3V)(1+\op (1)).
    }

For the Hamming graph $H(d,n)$ of an arbitrary
dimension $d$, we conjecture
that critical $p$ values are of the form
    \eqn{
    \lbeq{pcasympd>2}
    p=\sum_{i=1}^{\lceil d/3\rceil} a_i n^{-i} + \frac{\chs{\mu}}{n^{1+d/3}},
    }
where $\lambda$ is an arbitrary constant, and the coefficients
$a_i=a_i(d)$ are independent of $n$. Note that $p_c=\chs{1/\cn=
1/d(n-1)}$ corresponds to $a_i=a_i(d)=\chs{1/d}$ for all
$i\geq 1$, \ch{while \\ $p_c=\chs{1/(\cn-1)=1/(d(n-1)-1)}$, where
$d(n-1)-1$ is the {\it forward branching ratio of $H(d,n)$},
corresponds to $a_i=a_i(d)=(d+1)^i/d^{i+1}$ for all $i\geq 1$.} We
believe that, when $d$ is sufficiently large, there exists an $i$
such that $a_i(d)\neq \chs{1/d}$ \ch{and $a_i(d)\neq
(d+1)^i/d^{i+1}$.} In particular, if this is indeed true, then, for
$\vep=\Theta(V^{-1/3})$ and $d$ sufficiently large, the edge
probability $p=p_c+\frac{\vep}{\cn}$ is {\it not} the same as
$p=\frac{1+\vep}{\cn}$ \ch{or $p=\frac{1+\vep}{\cn-1}$.} For $d=2$,
however, these choices {\it do} \chs{asymptotically} agree, \chs{as explained
in \refeq{pcchires}--\refeq{critcomp2}.}

\ch{To explain why we believe \refeq{pcasympd>2} to hold, we note
that \cite{bchss1, bchss2} indeed gives that the critical window
consists of $p$ values given by $p=p_c+\chs{\mu} V^{-1/3}/\cn$, that
is $p=p_c+\chs{\mu} n^{-1-d/3}$ on \chs{$H(d,n)$}. Thus,
\refeq{pcasympd>2} follows \chs{for {\it all} $\mu$, as long as it
holds for one particular value of $p$ inside the critical window,
for example,} for $p=p_c$ defined in \refeq{pcdef}, for any
$\lambda$, for $d$ fixed and $n\rightarrow \infty$. Such asymptotic
expansions of critical values in terms of the vertex degree have
been established for the $n$-cube and for nearest-neighbour
percolation on ${\mathbb Z}^d$ in \cite{HS04a, HS04b}. These
expansions arise since the value $p_c$ satisfies an implicit
equation in terms of certain ``Feynman diagrams'' occurring in the
lace expansion analysis, and these diagrams can be proved to obey
asymptotic expansions that in turn imply that $p_c$ has an
asymptotic expansion. We expect that this part of the analysis in
\cite{HS04a, HS04b} can be extended to Hamming graphs, and will
allow one to compute the numerical values of $a_i(d)$.} The proof of
this conjecture would enable an extension to random subgraphs of
$H(d,n)$ of the phase transition description available for the
classical \chs{Erd\H{o}s-R\'enyi random graph.}

For $p$ inside the critical window, $|\Cmax|$ \ch{is} of the order
$V^{2/3}=n^{2d/3}$, \ch{as proved in \cite{bchss1, bchss2}.} Below
the critical window, we expect that the average cluster size
satisfies $\chi(p) \sim [\cn (p_c-p)]^{-1}$, while the maximum
cluster size satisfies
    \eqn{
    \lbeq{Cmaxeqsub}
    |\Cmax| \sim 2 \chi(p)^2 \log{\big(V/\chi(p)^3\big)} \quad \whp.
    }
Note that~\cite{bchss1} establishes in full only the upper bound part of
\refeq{Cmaxeqsub}, the corresponding best lower bound therein being
$|\Cmax| \geq \chi(p)^2/(3600\omega)$ \whp~for $\omega$ large
(cf.~\refeq{scalingsub}). (It is the upper bound, however, that
is relevant for locating the phase transition window.)
We anticipate that above the critical window
    \eqn{
    \lbeq{Cmaxgend}
    |\Cmax| \sim 2 \vep V \quad \whp,
    }
where $\vep=\cn (p-p_c)\gg V^{-1/3}$. \ch{Establishing the validity
of the asymptotics in \refeq{Cmaxgend} in full generality would
strengthen Theorem \ref{thm-maind2} to all $d\geq 1$ and all $p$
above the critical window.}

\section{Overview of the proof of Theorem \ref{thm-maind2}}

This section contains an extensive overview of the proof of our main result,
breaking it down into a number of key propositions and lemmas.
\ch{We start by describing the general philosophy of the proof.

From now on, we shall assume that $p=p_c+\vep/\chs{\cn}$, where $\vep\geq 0$.
As in \cite{bchss1}, the proof will be centered on the investigation
of the random variables
    \eqn{
    \lbeq{Zgeqdef}
    Z_{\sss \geq k}=\sum_{{\bf v}\in H(2,n)} I[|C({\bf v})|\geq k],
    }
the number of vertices in clusters of size at least $k$, for
appropriate values of $k$. In terms of these random variables, we
have that $|\Cmax|\geq k$ holds if and only if $Z_{\sss \geq k}\geq
1$. By proving sufficient concentration of measure for $Z_{\sss \geq
k}$, we are able to prove bounds on $|\Cmax|$. The whole proof
revolves around finding the right scales of $k$ to which we can
apply our arguments.

Specifically, we need two {\it different} scales. The first scale is the
smallest possible scale $k$ for which $\prob_p(|C({\bf v})|\geq
k)$ is very close to $2 \vep$. If indeed the duality principle holds
(see the discussion above \refeq{dualprin}), then, by \refeq{subprobkasy},
we expect that
    \eqn{
    \lbeq{supprobkasy}
    \prob_p(|C({\bf v})|\geq k)
    =\prob_p(|C({\bf v})|\geq k, {\bf v}\in \Cmax)+\prob_p(|C({\bf v})|\geq k, {\bf v}\not\in \Cmax)
    \sim 2\vep+
    \frac{1}{\sqrt{k}} e^{-k \chs{\vep^2}/2}.
    }
As a result, as soon as $k\gg \vep^{-2}$, we are led to
    \eqn{
    \lbeq{probpcluster}
    \prob_p(|C({\bf v})|\geq k)\sim 2\vep,
    }
so that also $\expec_p[Z_{\sss \geq k}]=V\prob_p(|C({\bf v})|\geq
k)\sim 2\vep V.$ Equation \refeq{probpcluster} follows from
Proposition \ref{lem-CCS} below. Assuming sufficient concentration
of measure for $Z_{\sss \geq k}$, we then obtain that $Z_{\sss \geq
k}\sim 2\vep V$ \whp, and, since $|\Cmax| \leq Z_{\sss \geq k}$ for
every $k$ for which $Z_{\sss \geq k}\geq 1$, we obtain the required
upper bound on $|\Cmax|$. Concentration estimates on $Z_{\sss \geq
k}$ are stated in Proposition \ref{lem-varZ} below.

The lower bound on $|\Cmax|$ is slightly more involved. Here we need
to find the {\it largest} possible $k$ for which we can prove that
$Z_{\sss \geq k}$ is concentrated around its mean $\expec_p[Z_{\sss
\geq k}]\sim 2\vep V$. To achieve this, we perform a so-called {\em
two-round exposure}. We first take $p_-<p$ such that
$p_-=p+o(\vep/\cn)$, and compare clusters in percolation with
parameter $p_-$ to suitable lower-bounding branching processes. Note
that such comparisons can only be applied when $k\ll \vep V$, so
these bounds are rather ``weak''. Subsequently, we ``sprinkle''
extra edges, so that the distribution of the final configuration is
that of percolation with parameter $p$. We prove that all the large
connected components in the $p_-$-configuration are, in fact,
\whp{}, joined all together by the sprinkled edges. We now explain
the steps in this argument in more detail.

Since $p_-<p$ and satisfies $p_-=p+o(\vep/\cn)$, all the
concentration results for $Z_{\sss \geq k}$ hold also for $Z'_{\sss
\geq k}$, the number of vertices in connected components of size at
least $k$ in the $p_--$percolation configuration. Furthermore, again
using the fact that $p_-=p+o(\vep/\cn)$, we have that
$\prob_{p_-}(|C({\bf v})|\geq k)\sim 2\vep$, so that, by our
concentration estimates, $Z'_{\sss \geq k}\sim 2\vep V$ \whp{} for
all $k\ll \vep V$. This establishes the necessary ``weak'' bounds on
connected components of size at least $k\ll \vep V$.

The $p$-configuration can be coupled to the $p_-$-configuration as
follows. Let $\eta>0$ be given by $p_-+(1-p_-)\chs{\eta/\cn} = p$. Then, make
each $p_--$vacant edge occupied with probability $\chs{\eta/\cn}$,
independently of all other vacant edges. We show that, for
appropriate choices of $\eta$ (and thus $p_-$) and $k\ll \vep V$,
the sprinkling procedure \whp{} connects {\it all} $p_-$-clusters of
size at least $k$ into one. It follows that $|\Cmax| \geq Z'_{\sss
\geq k}\sim 2\vep V$ \whp{}, establishing the lower bound. This part
of the proof makes crucial use of the fact that big components turn
out to be quite ``dense'', in the sense that they contain many
elements along most coordinate lines; details can be found in
Proposition \ref{lem-goodlines} below.

As explained above, the entire analysis revolves around a delicate
choice of the two different scales. We now present our precise
results, formulated in Propositions \ref{lem-CCS}, \ref{lem-varZ}
and \ref{lem-goodlines} below. We then use these propositions to
complete our proof of Theorem \ref{thm-maind2}.}

\begin{proposition}[The cluster tail]
\label{lem-CCS} Set $p=p_c +\frac{\vep}{\cn}$.  Let $V^{-1/3} \ll
\vep \ll 1$ as $V \to \infty$. Then, for every $\eta$ such that
$\vep^{-2} \ll \eta V\ch{\ll}\vep V$,
    \eqn{
    \prob_p(|C({\bf v})|\geq \eta  V)=2\vep (1+o(1)).
    }
\end{proposition}
Proposition~\ref{lem-CCS} consists of two parts, corresponding to the
upper and lower bounds. These are re-stated separately in
Section~\ref{S: explore} as
Lemmas~\ref{prop-BP} and~\ref{prop-couplingLB}, and proved in
Sections~\ref{S: explore} and~\ref{S: lower} respectively.
\smallskip

\noindent
The following proposition shows concentration of measure
for \ch{$Z_{\sss \geq k}$ for an appropriately chosen $k\gg \vep^{-2}$.}

\begin{proposition}[Concentration of the number of vertices in large
components \ch{of certain sizes}]
\label{lem-varZ}
Set $p=p_c +\frac{\vep}{\cn}$ and let
\ch{$V^{-1/3}(\log{V})^{1/3}\ll \vep \ll 1$.} \ch{Then there exists
$\vepnull$ satisfying $\vepnull \ll \vep$ such that, for every
$\delta>0$,}
    \eqn{
    \prob_p\Big(|Z_{\sss \geq \ch{\vepnull^{-2}}}-
    \expec_p[Z_{\sss \geq \ch{\vepnull^{-2}}}]|\geq \delta \vep V\Big)=o(1).
    }
\end{proposition}

\ch{The proof of Proposition \ref{lem-varZ} makes use of a somewhat
delicate second moment argument. We start by upper bounding the
variance of $Z_{\sss \geq N}$ and $Z_{\sss \geq 2N}-Z_{\sss \geq N}$
for specific values of $N$. These bounds are then combined to prove
that $Z_{\sss \geq \ch{\vepnull^{-2}}}$ is concentrated around its
mean, provided that $V^{-1/3}(\log{V})^{1/3}\ll \vep \ll 1$. The
proof can be found in Section \ref{sec-varest}, where we also show
that a possible choice for $\vepnull$ is $\vepnull=V^{-1/3}$, which
indeed satisfies $\vepnull^{-2}=V^{2/3}\gg \vep^{-2}$, since $\vep\gg
V^{-1/3}$. For the remainder of this section, we only need to know
that a suitable $\vepnull$ does indeed exist; its precise value is
irrelevant.}

Armed with Propositions~\ref{lem-CCS} and~\ref{lem-varZ}, we now
prove the upper bound on $|\Cmax|$ in Theorem \ref{thm-maind2}.

\smallskip

\begin{proofof}{the upper bound part of Theorem \ref{thm-maind2}}
We shall show that \whp, $|\Cmax|\le 2\vep V(1+o(1))$. Choose
$\vepnull$ as in Proposition~\ref{lem-varZ}.
By Proposition~\ref{lem-varZ}, the random variable $Z_{\sss \geq
\ch{\vepnull^{-2}}}$ is concentrated around $2\vep V$. \ch{In other words, the
number of vertices in connected components of size at least
$\vepnull^{-2}$ is close to $2\vep V$ \whp.} However, on the event
$\{Z_{\sss \geq \ch{\vepnull^{-2}}}\geq 1\}$, we have that
    \eqn{
    |\Cmax|\leq Z_{\sss \geq \ch{\vepnull^{-2}}},
    }
and so it follows that
    \eqn{
    |\Cmax|\leq 2\vep V(1+\op (1)).
    }
\end{proofof}

\ch{The following result is an easy corollary to
Proposition~\ref{lem-varZ}. It shows that, in fact, concentration of
measure holds for the number of vertices in clusters of size at
least $\eta V$, for {\it all} $\eta\ll \vep$ such that $\eta^3V\gg
1$. This will be required for the proof of the lower bound of
Theorem~\ref{thm-maind2}, as discussed in the proof overview above.}

\begin{corollary}[Concentration of the number of vertices in \ch{all} large
components] \label{lem-varZ-new}
Set $p=p_c +\frac{\vep}{\cn}$ and let
\ch{$V^{-1/3}(\log{V})^{1/3}\ll \vep \ll 1$.} Let $\eta=\eta (n)$
satisfy $\eta \ll \vep$ and $\eta^3 V \gg 1$. Then, for every
$\delta>0$,
    \eqn{
    \prob_p\Big(|Z_{\sss \geq \eta V}-
    \expec_p[Z_{\sss \geq \eta V}]|\geq \delta \vep V\Big)=o(1).
    }
\end{corollary}

\ch{Corollary \ref{lem-varZ-new} allows us to use the concentration
of $Z_{\sss \geq \eta V}$ for any appropriate $\eta$, thus
effectively removing the delicate choice of $\eta$ in Proposition
\ref{lem-varZ}. We shall see that Corollary \ref{lem-varZ-new}
follows from Propositions \ref{lem-CCS} and \ref{lem-varZ}, combined
with a simple first moment estimate.}

\begin{proof}
Let us choose $\eta$ satisfying both $\eta\ll \vep$ and $\eta^3 V\gg
1$, \ch{so that in particular $\eta V \gg \eta^{-2} \gg \vep^{-2}$.
Choose further $\vepnull$ as given in Proposition~\ref{lem-varZ}. We
shall assume that $\vepnull^{-2}\leq \eta V$; the proof when
$\vepnull^{-2}>\eta V$ is a simple adaptation of the argument below.}
By Proposition~\ref{lem-CCS}, for any fixed $\delta>0$,
    \eqan{
    \prob_p\Big(|Z_{\sss \geq \eta V}-2\vep V|\geq \delta \vep V\Big)
    &\leq \prob_p\Big(|Z_{\sss \geq \ch{\vepnull^{-2}}}-\expec_p[Z_{\sss \geq \ch{\vepnull^{-2}}}]|
    \geq \delta \vep V/3\Big)\nonumber\\
    &\qquad +\prob_p\Big(|Z_{\sss \geq \eta V}-Z_{\sss \geq \ch{\vepnull^{-2}}}|
    \geq \delta \vep V/3\Big),
    }
provided that $V$ is large enough. More precisely, the volume $V$
must be such that \eqn{
    \Big|\E_p[Z_{\sss \geq  \ch{\vepnull^{-2}}}]-2\vep V\Big|\leq \delta \vep V/3,
    }
or, equivalently (using transitivity),
    \eqn{
    \Big|\prob_p(|C({\bf v})|\geq  \ch{\vepnull^{-2}})-2\vep\Big|\leq \delta \vep /3.
    }
By Proposition~\ref{lem-varZ},
    \eqn{
    \lbeq{concZ}
    \prob_p\Big(|Z_{\sss \geq \ch{\vepnull^{-2}}}-\expec_p[Z_{\sss \geq \ch{\vepnull^{-2}}}]|
    \geq \delta \vep V/3\Big)=o(1).
    }
\ch{Further, since $\vepnull^{-2}\gg \vep^{-2}$, the Markov
inequality, together with the fact that $\vepnull\leq \eta$, yields
that, for every $\delta>0$,}
    \eqan{
    \lbeq{Markovbd}
    \prob_p\Big(|Z_{\sss \geq \eta V}-Z_{\sss \geq \ch{\vepnull^{-2}}}|
    \geq \delta \vep V/3\Big)
    &\leq \frac{3\expec_p[Z_{\sss \geq \ch{\vepnull^{-2}}}-Z_{\sss \geq \eta V}]}{\delta \vep V}\\
    &\ch{=\frac{3}{\delta \vep} \Big[\prob_p(|C({\bf v})|\geq  \vepnull^{-2})
    -\prob_p(|C({\bf v})|\geq  \eta V)\Big]}
    =o(1),\nonumber
    }
\ch{where the last equality follows from Proposition~\ref{lem-CCS}
together with the fact that $\vep^{-2} \ll \eta V\ll\vep V$ and
$\vep^{-2} \ll \vepnull^{-2}\ll\vep V$. Equation \refeq{Markovbd} thus
completes the proof.}
\end{proof}

It remains to establish a matching lower bound for $|\Cmax|$, and we
shall do this via  a ``sprinkling'' argument. (Sprinkling
is sometimes referred to as the ``two-round exposure'',
see~\cite[Chapter 1]{JLR}.) This part of our proof is based on
two results below. Before we state them, we need to introduce some
more notation.

For each $i$, the {\it $i$-th horizontal line} of $H(2,n)$
is defined to be the set $\{(i,x):
x=1, \ldots , n\}$ of vertices with first \chs{coordinate} $i$; similarly
the set $\{(x,i): x=1, \ldots , n\}$ of vertices with the second
\chs{coordinate} equal to $i$ constitutes the {\it $i$-th vertical line}. A
vertex belonging to a
given line is said to be an {\it element} of that line.

\begin{proposition}[Lower bound on the number of line elements
in a large cluster] \label{lem-goodlines} Set $p=p_c
+\frac{\vep}{\cn}$. There exists a constant $C > 0$ such that the
following holds. \ch{Fix $\vep,\eta$ satisfying $V^{-1/3} \ll \vep
\ll 1$, $\eta\ll \vep$, $\eta V\gg \vep^{-2}$ and $\eta V/ n \ge C
\log n$ for $n$ sufficiently large.} Then \whp{} for every cluster
of size at least $\eta V$, there are at least $\frac{3n}{4}$
horizontal lines each with at least $\eta V/(4n)$ elements contained
in the cluster.
\end{proposition}

The proof of Proposition~\ref{lem-goodlines} is deferred to Section
\ref{S: lower}. Assuming it holds, we now prove that the second
round exposure will join together every pair of large clusters
formed during the first round. In the following lemma, \ch{ for a
pair of sets of vertices $S_1, S_2$, we use the notation $S_1 \conn
S_2$ to denote the event that $S_1,S_2$ are joined together.
\chs{We also write $S_1 \nc S_2$ to denote that $S_1,S_2$ are not joined together.}}

\begin{lemma}[Sprinkling]
\label{lem-sprinkling} Set $p=p_c +\frac{\vep}{\cn}$. Choose
$V^{-1/3} \ll \vep \ll 1$. \ch{Let $\eta=\sqrt{\vep} V^{-1/6}$, and
let $S_1,S_2$ be disjoint sets of vertices both containing at least
$\eta V/(4n)$ elements of at least \chs{$3n/4$} horizontal lines
(possibly different lines for $S_1$ and $S_2$). Then
    \eqn{
    \prob_{\eta/\cn} (S_1 \conn S_2) \geq 1-
    o\Big(\frac{V^{-1/3}}{\vep}\Big).
    }
}
\end{lemma}
\begin{proof}
Choose two disjoint vertex sets $S_1$ and $S_2$ each containing at
least $\eta V/4$ elements in at least \chs{$3n/4$} horizontal lines.
Then $S_1$ and $S_2$ must have at least \chs{$n/2$} such lines in
common, that is {\it both} $S_1$ and $S_2$ contain at least $\eta
V/(4n)=\eta n/4$ elements of these lines. \ch{Note that, since
$\vep\gg V^{-1/3}$, $\eta=\sqrt{\vep} V^{-1/6}$ and $V=n^2$, we have
$\eta V/n=\sqrt{\vep}V^{1/3}\gg 1$.} Along the shared good lines,
there are at least $(\eta n)^2/16$ edges with one endpoint in $S_1$
and the other in $S_2$. All of these edges will be occupied
independently under $\prob_{\eta/\cn}$, so
    \eqn{
    \prob_{\eta/\cn} (S_1 \nc S_2) \leq \Big(1-\frac{\eta}{2(n-1)}\Big)^{n(\eta n)^2/\chs{32}}.
    }
Using the inequality $1-x\leq e^{-x}$ and the fact that $\vep \gg
V^{-1/3}$,
    \eqn{
    \ch{\prob_{\eta/\cn} (S_1 \nc S_2) \leq e^{-\eta^3 n^2/\chs{64}}=e^{-\vep^{3/2}V^{1/2}/\chs{64}}\ll
    \Big(\vep V^{1/3}\Big)^{-1}=o(1),}
    }
which completes the proof.
\end{proof}

We can now do the lower bound part of Theorem~\ref{thm-maind2}.

\begin{proofof}{the lower bound part of Theorem~\ref{thm-maind2}}
\ch{We choose $\eta=\sqrt{\vep} V^{-1/6}$, as in
Lemma~\ref{lem-sprinkling}, and note that the results in
Proposition~\ref{lem-goodlines} apply to this choice of $\eta$.
Indeed, since $\vep\gg V^{-1/3}$, we have that $\eta=\vep
/(\sqrt{\vep V^{1/3}})\ll \vep$ and $\eta V=\sqrt{\vep} V^{5/6} \gg
\vep^{-2}$. Finally, once again using $\vep\gg V^{-1/3}$, we have
$\eta V/ n=\sqrt{\vep} V^{1/3}\gg V^{1/6}\ge C \log n$ for $n$
sufficiently large.

We define $p_-$} by the relation
    \eqn{
    p_-+(1-p_-)\frac{\eta}{\cn} = p.
    }
Note that every configuration with edge probability $p$ can be obtained
in a unique way as follows. First construct a configuration by
throwing in edges independently of one another with probability $p_-$;
subsequently, ``sprinkle'' extra edges
with probability $\frac{\eta}{\cn}$, independently
of one another and of the $p_-$ configuration. In the final
configuration, an edge is occupied precisely when it is occupied in
either the $p_-$ configuration, or when it is an edge that is
added during the sprinkling procedure. Since $\eta\ll \vep$,
    \eqn{
    p_-=p+o \bigg (\frac{\vep}{n} \bigg).
    }
Let $Z'_{\sss \geq \eta V}$ denote the number of vertices in
connected components of size at least $\eta V$ in the $p_-$
configuration. Since $\delta$ in Corollary~\ref{lem-varZ-new} is
arbitrary, it implies that $Z'_{\sss \geq \eta V}=2\vep V(1+\op(1))$
after the first round of exposure; and by
Proposition~\ref{lem-goodlines} \whp{} every cluster of size at
least $\eta V$ includes at least $\eta V/(4n)$ elements in at least
$\frac{3n}{4}$ lines. Thus, under the measure $\prob_{p_-}$, \whp{}
there are at most $\frac{2\vep}{\eta}(1+o(1))$ connected clusters of
size at least $\eta V$, and each of these connected components
contains at least $\eta V/(4n)$ elements in at least $\frac{3n}{4}$
lines.

It now suffices to prove that \whp{} the subsequent sprinkling
procedure (second round of exposure) joins together every pair of
clusters of size at least $\eta V$.
Indeed, if this is the case, then after the sprinkling we
end up with a single connected component of size at least
    \eqn{
    \lbeq{Z'bd}
    Z_{\sss \geq \eta V}'\geq 2\vep V(1+\op(1)).
    }

Let ${\bf v}_1$ and
${\bf v}_2$ be two vertices such that $C({\bf v}_1) \not = C({\bf
v}_2)$, and $|C({\bf v}_1)|\geq \eta V$ and $|C({\bf
v}_2)|\geq \eta V$. Let us take $S_1=C({\bf v}_1)$ and $S_2=C({\bf
v}_2)$. By Proposition~\ref{lem-goodlines},
we may assume that for both $S_1$ and $S_2$ one can find at least
$\frac{3n}{4}$ lines (not necessarily the same ones for $S_1$ and
$S_2$) each with at least $\eta V/(4n)$ elements in $S_1$ and $S_2$. Then,
by Lemma~\ref{lem-sprinkling},
    \eqn{
    \prob_{\eta/\cn} (S_1 \nc S_2)=
    o\Big(\frac{V^{-1/3}}{\vep}\Big).
    }
But \whp{} there are at most \ch{$\Big(\frac{2\vep}{\eta}\Big)^2(1+o(1))
=O(\vep V^{1/3})$} \chs{distinct} choices for
$\chs{C({\bf v}_1)}$ and $\chs{C({\bf v}_2)}$ with $|C({\bf v}_1)|\geq \eta V$
and $|C({\bf v}_2)|\geq \eta V$, and so a simple union bound implies that
after sprinkling \whp{} all connected components of size at
least $\eta V$ are connected. By~\refeq{Z'bd}, this completes the
proof.
\end{proofof}

\smallskip

\noindent

The remainder of the paper is organised as follows.
In Section \ref{S:
branching} we prove auxiliary results relating to \ch{the tails of
the total progeny of binomial Galton-Watson processes.} \ch{Section
\ref{S: comparison} contains proofs of Propositions~\ref{lem-CCS}
and \ref{lem-goodlines}; therein we investigate the structure of
percolation clusters (cluster tails and the number of elements per
\chs{coordinate} line) by comparing them to binomial Galton-Watson
processes.} \ch{Finally, in Section \ref{S: conc}, we establish
concentration of measure for the number of vertices in large
clusters, thus proving Proposition~\ref{lem-varZ}.}

\section{Total progeny of a Galton-Watson process}
\label{S: branching}

This section brings together some useful results from the theory of branching
processes, which will play a key role at various stages in our proofs.

We consider a standard Galton-Watson process whose offspring
distribution $Z$ is a binomial $\Bi (N,p)$, where $N \in \bbN$ and
$p \in [0,1]$ is the Hamming graph edge probability.
We assume that with probability 1 the process begins with one
individual. We write $\prob_{N,p}$ for the probability measure
corresponding to this process (implicitly assuming an underlying
sample space and $\gs$-field).

Let $F$ be the total progeny or family size.
Our aim is to prove the following three results concerning the
distribution of $F$.
Proposition~\ref{prop.branch-1} compares the
distribution of $F$ under the measures $\prob_{N,p}$,
$\prob_{\tN,p}$ for different values of $N$ and $\tN$.
Proposition~\ref{prop.branch-2}
estimates the probability that the value of $F$ is between $\ell$ and
$2\ell$, for some (large) integer
$\ell$. Proposition~\ref{prop.branch-3} estimates the
probability that $F$ takes a value at least $\ell$ for some large integer $\ell$.

\begin{proposition}[Tails of total progeny in two binomial branching
processes]
\label{prop.branch-1}
Let $\ell \in \bbN$. Suppose that $N \in \bbN$ and $\tN = \tN
(N)$ satisfy $N\geq \tN$. 
Further, assume that $\vep =
Np-1$ is such that $\vep\to 0$ and $\vep \ge N^{-2/3}$;
and that $\tilde\vep=\tN p-1>0$ and \ch{$|\vep-\tvep|=o(\vep)$} as
$N \to \infty$.
Then, for some constant $C >
0$, as $N \to \infty$,
    \eqn{
    \label{ineq.branch-1b}
    |\prob_{N, p}(F\geq \ell)-\prob_{\tN, p}(F\geq \ell)|
    \leq C
\Big(|\vep-\tvep|+\frac{1}{N\ell^{1/2}}+\frac{1}{\ell^3} \Big).
    }
\end{proposition}

\begin{proposition}[Bounds on the total progeny distribution]
\label{prop.branch-2}
Let $N,\ell \in \bbN$. Suppose that $\vep =
Np-1$ satisfies $\vep \to 0$ and
$\vep   \ge N^{-2/3}$ as $N \to \infty$. Then, for some
constant $C > 0$, as $N \to \infty$,
\eqn{
    \lbeq{ineq.branch-2}
    \prob_{N,p} (F \in [\ell, 2\ell])\leq \frac{C}{\sqrt{\ell }}.
    }
\end{proposition}

\begin{proposition}[Tails of the total progeny near criticality]
\label{prop.branch-3}
Let $N,\ell \in \bbN$. Suppose that $\vep =
Np-1$ satisfies $\vep \to 0$ and
$\vep   \ge N^{-2/3}$ as $N \to \infty$.
Then, as $N \to \infty$,
\eqn{
    \lbeq{ineq.branch-3}
    \prob_{N,p} (F \ge \ell )=2 \vep + O(\vep^2)+ O\Big ( \frac{1}{\sqrt{\ell }}\Big).
    }
\end{proposition}

We note that with a little care and minor modifications,
Propositions~\ref{prop.branch-1}--\ref{prop.branch-3} could be
extended to the case where $\vep$ is a positive constant
(i.e. strictly above the critical window).
In the corresponding statement,~\refeq{ineq.branch-3} would have
$\zeta_{1+\vep}$ instead of $2 \vep$, where, as before, $\zeta_{\gl}$
denotes the \chs{survival} probability of a Poisson Galton-Watson
process with mean family size $\gl$.

Our proofs of Propositions~\ref{prop.branch-1} and~\ref{prop.branch-2}
will make use of
the well\chs{-}known {\em Otter-Dwass formula}, which describes the distribution of the
total progeny of a branching process, see ~\cite{dwass,otter}. We
begin by stating a special case of this formula (due to
Otter) for a branching process starting with 1 individual.
(The formula was later extended by Dwass to a process starting with
$r$ individuals, for arbitrary $r \in \bbN$, but we do not make use of
the extension here.)

\begin{lemma}[Otter-Dwass formula]
\label{lem-OD}
Let $Z_1, Z_2, Z_3, \ldots $ be \iid{} \rv{s} distributed as $Z$. Let $\prob$
denote the Galton-Watson process measure.
For all $\chs{k} \in \bbN$,
$$\prob (F =k) =\frac{1}{k} \prob (\sum_{i=1}^k Z_i = k-1).$$
\end{lemma}

\smallskip

\noindent
We now prove each of Propositions~\ref{prop.branch-1}--\ref{prop.branch-3}
in turn.

\begin{proofof}{Proposition~\ref{prop.branch-1}}
It will be convenient for us to introduce new parameters $\gl = pN$
and $\tilde{\gl} = p\tN$.
By assumption $\gl,\tilde{\gl} > 1$, so that both branching
processes are supercritical. The total progeny size in the $\Bi
(\tN,p)$ process will be denoted by $\tilde{F}$.

By Lemma~\ref{lem-OD}, for each $k \in \bbN$,
    \eqn{
    \lbeq{O-D}
    \prob_{N,p} (F =k) =\frac{1}k \prob (\sum_{i=1}^k Z_i =
    k-1),}
where the $Z_i$ are \iid{} $\Bi (N,\gl/N)$. We now investigate the
asymptotics of the formula~(\ref{O-D}) as 
$N \to \infty$, for large integers $k=k(N)$. Our aim
is to obtain estimates for $\prob_{N,p} (F =k)$ sharp enough for the
errors to be summable.

In outline, our calculation is as follows. First we demonstrate that, if
$k\ge  C_0\vep^{-2}\log (1/\vep)$ for a sufficiently large constant $C_0$, then
    \eqn{
    \lbeq{aim1-31}
    \prob_{N,p} (k \le F< \infty )\leq k^{-4}.
    }
Next we show that \refeq{aim1-31} implies an
identical upper bound for $\prob_{\tN,p} (\tilde{F} =k)$.
Subsequently, we prove that there exists a constant $\tilde{C} > 0$ such that,
if $k\le C_0\vep^{-2}\log (1/\vep)$, then
    \begin{align}
    \lbeq{aim2-31}
    |\prob_{N,p} (F=k)-\prob_{\tN,p} (\tilde{F}=k)| &\le
    \frac{\tilde{C}}{k^{3/2}}\exp \Big (-\frac{(k-1)\vep^2}{\chs{4}}
    \Big )
    \Big |(1+k\vep)|\vep-\tvep|
    +\frac{k}{N^3}+ \frac{\vep}{N}+ \frac{1}{kN}\Big|,
    \end{align}
We can then sum the errors in \refeq{aim2-31} and \refeq{aim1-31}
to show that, for any $\ell \in \bbN$,
    \eqn{
    \lbeq{aim2b-31}
    |\prob_{N,p} (\ell\leq F<\infty)-
    \prob_{\tN,p} (\ell\leq \tilde{F}<\infty)| \le
    C\Big(|\vep-\tvep| +\frac{1}{N\ell^{1/2}}+\frac{1}{\ell^3}\Big).
    }
Finally, since $\gl > 1$, we need to estimate $\prob_{N,p} (F =
\infty)$ and $\prob_{\tN,p} (\tilde{F} =\infty)$; we shall show that
    \eqn{
    \lbeq{aim3-31}
    \big|\prob_{N,p} (F =\infty)-\prob_{\tN,p} (\tilde{F} =\infty)\big|
    \leq C|\vep-\tvep|.
    }
Combining the last two estimates yields
Proposition~\ref{prop.branch-1}. As we shall see below,
several steps of our proof will also play a role
in proving Propositions~\ref{prop.branch-2} and \ref{prop.branch-3}.

Let us make a start on the details.
To show \refeq{aim1-31}, we note that \refeq{O-D} implies
    \eqn{
    \lbeq{O-Drep}
    \prob_{N,p} (F =k) \leq \frac{1}k \prob (\sum_{i=1}^k Z_i \leq k),}
where $\sum_{i=1}^k Z_i\sim \Bi(kN, p)$. But if $Z\sim \Bi(n,p)$, then
(see for instance~\cite{Jans02})
    \eqn{
    \lbeq{Binbd}
    \prob(\chs{Z}\leq np-t) \leq e^{-\frac{t^2}{2(np+\frac{t}{3})}}.
    }
Applying~(\ref{Binbd}) with $n=kN$, $p=\frac{1+\vep}{N}$ and $t=np-k=k\vep$
(and also using our assumption that $\vep < 1$), we obtain that
    \eqn{
    \lbeq{Binbdres}
    \prob_{N,p} (F =k)\leq \frac{1}k e^{-\frac{k \vep^2}{2(1+\frac{4\vep}{3})}}
    \leq \frac{1}k e^{-\frac{k \vep^2}{5}}.
    }
But if $k\ge  C_0\vep^{-2}\log (1/\vep)$ for a sufficiently large $C_0 > 0$
then
    \eqn{
    \frac{1}k e^{-\frac{k \vep^2}{4}}\leq k^{-4},
    }
so~\refeq{aim1-31} follows.

We now show that \refeq{aim2-31} holds for $k\le  C_0\vep^{-2}\log (1/\vep)$.
Clearly,~\refeq{O-D} implies that, for each $k \in \bbN$,
    \eqn{
    \lbeq{totpronexpl}
    \prob_{N,p}(F =k)=\frac{1}k {{k N}\choose {k-1}}\Bigg (\frac{\gl}{N}
    \Bigg)^{k-1} \Bigg (1-\frac{\gl}{N}\Bigg)^{k N-k+1}.
    }
By Stirling's formula,
    $$(m)_r \= m(m-1) \chs{\cdots} (m-r+1) = m^r  \exp \Big ( -\frac{r^2}{2m} -
    \frac{r^3}{6m^2} + O\Big (\frac{r^4}{m^3}\Big )\Big ).$$
Applying the above approximation with $m=kN$ and $r=k-1$ (and noting
that $\frac{r^4}{m^3}=O(\frac{k}{N^3})$),
we arrive at
    \begin{align*}
    \prob_{N,p} (F=k) &=
    \frac{1}k \frac{(k\lambda)^{k-1}}{(k-1)!} \exp \Big (
    -\frac{(k-1)^2}{2kN} -
    \frac{(k-1)^3}{6k^2N^2} + O\Big (\frac{k}{N^3}\Big )\Big )
    \Big (1-\frac{\gl}{N}\Big )^{k N-k+1}\\
    &=\frac{(k\gl)^{k-1} }{k!}
    \exp \Big (
    -\frac{k}{2N} + \frac{1}{N}-\frac{1}{2kN}-
    \frac{k}{6N^2} + O\Big (\frac{1}{N^2}+ \frac{k}{N^3}\Big )\Big )
    \Big (1-\frac{\gl}{N}\Big )^{k N-k+1}.
    \end{align*}
Observe that
    \begin{align*}
    \Big (1-\frac{\gl}{N}\Big )^{k N-k+1}&
    =\exp \big((kN -k+1) \log (1-\frac{\gl}{N})\big)\\
    &=\exp \Big(-\gl k + \frac{k \gl}{N}-\frac{\gl}{N} -\frac{\gl^2 k}{2N} + \frac{\gl^2 k}{2N^2}-
    \frac{\gl^3 k}{3N^2}+ O(\frac{1}{N^2}+ \frac{k}{N^3})\Big),
    \end{align*}
so that
    \begin{align}
    \prob_{N,p} (F=k) 
    \label{qxdef}
    &= \frac{k^{k-1}e^{-(k-1)} e^{-\gl }}{k!}\exp ((k-1)f(\gl)) \exp \Big (
    -\frac{k(\gl -1)^2}{2N} + \frac{1-\gl}{N}-\frac{1}{2kN}
    \Big )\nonumber\\
    & \qquad\times \exp \Big (  \frac{k}{N^2}g(\lambda)+ O\Big
    (\frac{k}{N^3}+ \frac{1}{N^2}\Big )\Big),
    \end{align}
where
    \eq
    f(\gl)=\log \gl-(\gl-1), \qquad g(\gl)=-\frac16+ \frac{\gl^2}{2}-\frac{\gl^3}{3}.
    \en
The Taylor expansion for $|\lambda-1|$ small gives
    \eq
    \lbeq{flTaylor}
    f(\gl)= -\frac{(\lambda-1)^2}{2}+O(|\lambda-1|^3).
    \en
But $\lambda-1=\vep$ and $k\le  C_0 \vep^{-2}\log (1/\vep)$, where
$N^{-2/3} \ll \vep =o(1)$, and so
$k|\lambda-1|^3=o(1)$, uniformly for all such $k$.
The other error terms can be bounded similarly, and hence,
uniformly for $k\le  C_0\vep^{-2}\log (1/\vep)$,
    \eq
    \lbeq{probBPasympt}
    \prob_{N,p} (F=k)= (1+o(1))
    \frac{k^{k-1}e^{-k}}{k!}\exp \Big (-\frac12 (k-1) (\gl-1)^2\Big).
    \en

We now compare $\prob_{N,p} (F=k)$ to $\prob_{\tN,p} (\tilde{F}=k)$
for $k\le  C_0 \vep^{-2}\log (1/\vep)$.
Write $\gl /N = \tilde{\gl}/\tN$, where
$\tilde{\gl} = \tN \gl/N<\gl$.
Then a calculation similar to the one above shows that
    \begin{align*}
    \prob_{\tN,p} (\tilde{F}=k) &=
    \frac{k^{k-1}e^{-\tilde{\gl}-(k-1)}}{k!}
    \exp \Big ((k-1) f(\tilde{\gl})\Big ) \\
    &\qquad \times \exp \Big (
    -\frac{k(\tilde{\gl} -1)^2}{2\tN} +
    \frac{1-\tilde{\gl}}{\tN}-\frac{1}{2k\tN }
    + \frac{k}{\tN^2}g(\tilde{\gl}) + O\Big (\frac{k}{\tN^3}+
    \frac{1}{\tN^2} \Big )\Big )\Big ).
    \end{align*}
By assumption, $\vep = \gl -1>0$ and $\tvep = \tilde{\gl}-1>
0$. Further,
    \begin{align}
    f(\gl)-f(\tilde{\gl})
    = (\tvep-\vep) f'(1+s),
    \end{align}
for some $s\in (\tvep, \vep)$, where
    \eq
    |f'(1+s)|=\frac{s}{1+s}\leq s, \quad s\geq 0,
    \en
and so
    \eq
    f(\gl)-f(\tilde{\gl})=O(\vep|\vep-\tvep|).
    \en
We deduce that, for $\vep \ge N^{-2/3}$ and all $k\le C_0\vep^{-2}\log (1/\vep)$,
    \begin{align}
    \lbeq{diffProbbd}
    |\prob_{N,p} (F=k)-\prob_{\tN,p} (\tilde{F}=k)| & =
    \frac{(k/e)^{k-1}e^{-\gl}}{k!} \exp ((k-1)f(\gl))
    \big | \exp(x)-\exp(y)\big|
    \end{align}
where
    \begin{align*}
    x&=-\frac{k\vep^2}{2N}+\frac{k}{N^2}g(\gl)+
    O\Big (\frac{k}{N^3}+ \frac{\vep}{N}+ \frac{1}{kN}\Big ),\\
    y&=(\vep - \tvep)
    -\frac{k\tvep^2}{2\tN} + \frac{k}{\tN^2}g(\tilde{\gl})+ O(k\vep|\vep-\tvep|)
    + O\Big (\frac{k}{\tN^3}+ \frac{\tvep}{\tN}
    + \frac{1}{k\tN}\Big).
    \end{align*}
Since
$k\le  C_0\vep^{-2}\log (1/\vep)$, $N^{-2/3} \ll \vep =o(1)$ and
$\tvep =o(\vep)$, $x=o(1)$
and also all contributions to $y$ are $o(1)$,
except \ch{possibly} for the term $k\vep|\vep-\tvep|$.
\ch{Now, for some constant $C$,
    \eqn{
    \lbeq{Taylorexpfirst}
    |\exp(x)-\exp(y)|\leq C|x-y|e^{|x|\vee |y|},
    }
\chs{where, for $u,v\in {\mathbb R}$,  $u\vee v=\max\{u,v\}$.}
Note that $x=o(1)$ and, since $|\vep-\tvep|=o(\vep)$, we have that
$y=o(1)+o(k\vep^2)$. As a result, we obtain that,
for $N$ sufficiently large,
    \eqn{
    \lbeq{Taylorexp}
    |\exp(x)-\exp(y)|\leq C|x-y|e^{(k-1)\vep^2/4}.
    }
Since further} $N^{-2} - \tN^{-2} = O (N^{-2}|\vep
-\tvep|)$, the contribution to $|x-y|$ due to the term
$g(\gl)k/N^2-g(\tilde{\gl})k/\tN^2$ is $O(kN^{-2}|\vep
-\tvep|)=o(k\vep|\vep-\tvep|)$, which gives
    \eqn{
    |x-y|\leq C(1+k\vep)|\vep-\tvep|+O\Big(\frac{k}{N^3}+ \frac{\vep}{N}+ \frac{1}{kN}\Big).
    }
Hence,
\chs{combining \refeq{diffProbbd}, \refeq{Taylorexp}, \refeq{flTaylor},
for all $k \le C_0 \vep^{-2}\log (1/\vep^3)$, we arrive at
\refeq{aim2-31}.}

Summing the estimates \refeq{aim1-31} and \refeq{aim2-31} over $k\geq \ell$,
    \begin{align*}
    &|\prob_{N,p} (\ell \le F<\infty )-\prob_{\tN,p}
    (\ell \le \tilde{F}<\infty )| \\
    &\quad \le  C \sum_{k \ge \ell}
    \Big(
(1+k\vep)|\vep-\tvep|
    +\frac{k}{N^3}+ \frac{\vep}{N}+ \frac{1}{kN}\Big)
    \frac{\exp (-(k-1) \vep^2/\ch{4})}{k^{3/2}}+\sum_{k \ge \ell}k^{-4}.
    \end{align*}
The final contribution is $O(\ell^{-3})$; for the remaining terms observe that
    \eqn{
    \sum_{k \ge \ell} k^{-a} \exp (-(k-1) \vep^2/\ch{4})
    \leq \begin{cases}
    C\ell^{1-a} &\text{for }a>1,\\
    C\vep^{-2-2a} &\text{for }a<1,
    \end{cases}
    }
which yields (with a suitably adjusted value of $C$)
    \begin{align}
    |\prob_{N,p} (\ell \le F <  \infty )-\prob_{\tN,p}
    (\ell \le \tilde{F} < \infty )|&\leq C\Big(
    |\vep -\tvep|+\frac{1}{\vep N^3}+\frac{\vep}{N \ell^{1/2}}
    +\frac{1}{\ell^{3/2} N}+\frac{1}{\ell^3}\Big)\nonumber\\
    &\leq C\Big(|\vep -\tvep|+\frac{1}{N\ell^{1/2}}+\frac{1}{\ell^3}\Big).
    \end{align}

To prove \refeq{aim3-31},
we need to estimate
$|a-\tilde{a}|$, where
$a=\prob_{N,p} (F< \infty)$ and $\tilde{a}=\prob_{\tN,p}
(\tilde{F}<\infty)$. The quantities $a,\tilde{a}$ respectively are the
smallest positive roots of the equations
    \eqn{
    \label{eq.extinct}
    a = (1 + \frac{\gl}{N} (a-1))^N,
    \qquad \text{and}\qquad
    \tilde{a} = (1 + \frac{\gl}{N}
    (\tilde{a}-1))^{\tN}.
    }
Using the convexity of probability generating
functions and the supercriticality of the branching
processes in question, the equations in~\eqref{eq.extinct} each have
precisely one root $a$ and $\tilde{a}$ respectively in the interval $[0,1)$.

The proof is divided into two main steps. In the first step, we prove that
$1-a=2\vep+O(\vep^2)$, which also implies
that $|a- \tilde{a}|=o(\vep)$
when $|\vep-\tvep|=o(\vep)$, so that we may use a Taylor expansion.
In the second main step, we prove that $|a- \tilde{a}|\leq C|\vep-\tvep|$.

To prove that $1-a=2\vep+O(\vep^2)$,
we expand the right hand side of~(\ref{eq.extinct}) to obtain
    \begin{align*}
    a-1 & = {N \choose 1} \frac{\gl}{N} (a-1) + {N \choose 2} \Big
    (\frac{\gl}{N} \Big )^2 (a-1)^2 + O(|a-1|^3)\\
    1 & = \gl + \frac{N-1}{2N}\gl^2 (a-1) + O(|a-1|^2)\\
    1 & = 1+\vep + \frac{N-1}{2N}(1+\vep)^2 (a-1)+ O(|a-1|^2),
    \end{align*}
so that
    \begin{align*}
    (1-a) (1 + 2 \vep + \vep^2-N^{-1} -2 N^{-1} \vep + O(N^{-1} \vep^2)) =
    2 \vep + O(|1-a|^2),
    \end{align*}
and so, again using that $\vep\geq N^{-2/3}$,
    \begin{align}
    \lbeq{aasympt}
    1-a = 2 \vep+O(\vep^2).
    \end{align}
To prove that $|a-\tilde{a}|\leq C|\vep-\tvep|$, we use that
    \eqn{
    a-\tilde{a}= (1 + \frac{\gl}{N} (a-1))^N-(1 + \frac{\gl}{N}
    (\tilde{a}-1))^{\tN}
    =f_{\tN}(a) -f_{\tN}(\tilde{a}) +
    f_{\tN}(a)\Big((1 + \frac{\gl}{N} (a-1))^{N-\tN}-1\Big),
    }
where $f_{\tN}(x)=(1 + \frac{\gl}{N}(x-1))^{\tN}.$
Note first that
    \eqn{
    (1 + \frac{\gl}{N} (a-1))^{N-\tN}-1
    =O\Big(\frac{N-\tN}{N}(1-a)\Big).
    }
Further, since $|a-\tilde{a}|=o(|a-1|)$, we have that $f_{\tN}(a)\leq 2$.
Also,
    \begin{align}
    f_{\tN}'(x)&=\frac{\tN \lambda}{N}\big(1 + \frac{\gl}{N}
    (x-1)\big)^{\tN-1},\\
    f_{\tN}''(x)&=\frac{\tN \lambda}{N}\frac{(\tN-1)\lambda}{N}\big(1 + \frac{\gl}{N}
    (x-1)\big)^{\tN-2},
    \end{align}
and it is not hard to see that $f_{\tN}''(x)=O(1)$ uniformly
for $x\in [\tilde{a}, a]$. Hence
    \eqn{
    a-\tilde{a}
    =(a-\tilde{a})f_{\tN}'(\tilde a)+O\Big(\frac{N-\tN}{N}(1-a)\Big)+O\big((a-\tilde{a})^2\big),
    }
so that
    \eqn{
    a-\tilde{a}=O\Big(\frac{(N-\tN)(1-a)}{N|1-f_{\tN}'(\tilde a)|}\Big)
    +O\Big(\frac{(a-\tilde{a})^2}{|1-f_{\tN}'(\tilde a)|}\Big).
    }
A closer inspection of $f_{\tN}'(x)$ yields that
$f_{\tN}'(\tilde a)-1=\vep+o(\vep)$, so that
    \eqn{
    |a-\tilde{a}|\leq O\Big(\frac{|N-\tN|}{N}\Big)
    +O\Big(\frac{(a-\tilde{a})^2}{1-a}\Big)\leq C|\vep-\tvep|.
    }
This completes the proof of Proposition \ref{prop.branch-1}.
\end{proofof}

\smallskip

\noindent
\begin{proofof}{Proposition~\ref{prop.branch-2}}
By \refeq{probBPasympt}, for all $k \le C_0\vep^{-2} \log(1/\vep)$
    \begin{align*}
    \prob_{N,p} (F=k) &=
    (1+o(1))\frac{(k/e)^{k-1}e^{-\gl}}{k!}\exp \Big(-\frac12 (k-1)
    \vep^2\Big).
    \end{align*}
Also (provided $C_0$ is large enough) for $k \ge C_0\vep^{-2} \log(1/\vep)$,
 \begin{align*}
    \prob_{N,p} (F=k) \le k^{-4}.
    \end{align*}
Summing over $\ell \le k \le 2 \ell$ we obtain
    \begin{align*}
    \prob_{N,p} (\ell \le F \le 2 \ell)
    & \le C  \sum_{\ell \le k \le 2 \ell} \Big(\frac{1}{k^{3/2}}
    e^{-k\vep^2/2}+k^{-4}\Big) \le \frac{C}{\ell^{1/2}},
    \end{align*}
where the constant $C$ was adjusted within the final inequality.
\end{proofof}

\smallskip

\noindent
\begin{proofof}{Proposition~\ref{prop.branch-3}}
We have
    \eqn
    {\lbeq{probsplitProp32}
    \prob_{N,p} (F \geq \ell) = 1-\prob_{N,p} (F < \infty) +\prob_{N,p} (\ell\leq F < \infty).
    }
By \refeq{aasympt}, the term $1-\prob_{N,p} (F < \infty)$ is
$2\vep + O (\vep^2)$. Calculations similar to those in the proof of
Proposition~\ref{prop.branch-2} show that the final term is bounded by
$O(\ell^{-1/2})$, which completes the proof.
\end{proofof}

\smallskip

\section{Comparisons to branching processes}
\label{S: comparison} \ch{In this section, we use comparisons to
branching processes and concentration of measure techniques to study
the cluster tail probabilities (cf. Proposition \ref{lem-CCS}), as
well as the cluster structure, specifically, the number of vertices
per line, of large clusters (cf. Proposition \ref{lem-goodlines}).

This section is organised as follows. In Section~\ref{S: explore},
we describe a {\it cluster exploration} procedure, state key
estimates for the tails of the cluster size distribution, and prove
the upper bound part of Proposition~\ref{lem-CCS}. In
Section~\ref{S: upper} we establish an upper bound on
the number of elements per line in a large cluster; this result is a
crucial ingredient in the proof of Proposition~\ref{lem-goodlines}.
Section \ref{S: lower} contains a proof of the lower bound part of
Proposition~\ref{lem-CCS}, as well as a proof of
Proposition~\ref{lem-goodlines}.}

\subsection{Component exploration and strategy of proof}
\label{S: explore}

We take an initial vertex ${\bf v}_0= (x_0,y_0)$ and {\it explore}
its cluster, $C({\bf v}_0)$, by {\it exploring} the vertices in that
cluster successively one at a time, in a breadth-first order.
Exploring a vertex $(x,y)$ means that we consider all the edges
$(x,j)$ for $j \not = y$ in the order of increasing $j$, and decide
for each one in turn if it is open with probability $p$ or closed
with probability $1-p$; then we do the same for the edges $(i,y)$
for $i \not = x$ in the order of increasing $i$. \chs{Note that,
until the moment all available vertices in the cluster have been
explored, the number of explored vertices at time $t$ is equal to
$t$.}

Let us introduce
colours as follows. At time $t$, all vertices that have not yet been
explored and are not yet contained in $C({\bf v}_0)$ are
{\it white}. All unexplored vertices connected to ${\bf v}_0$
(that is, included in $C({\bf v}_0)$) at time
$t$ are {\it green}. All explored vertices are
{\it red}. (Thus, in particular, at time 0
all vertices are white except for ${\bf v}_0$, which is green.)
In fact, we need to modify this exploration process slightly as
follows: when exploring a green vertex we only consider those of its
edges where the other endpoint of the edge is white. If such an edge
is found to be open, then we colour its other endpoint green.

Let $C_t ({\bf v}_0)$ be the set of vertices included in the cluster
of ${\bf v}_0$ by the time $t$. Let also $G_t ({\bf v}_0)$ be the set of green
vertices in the cluster at time $t$. Thus $C_t ({\bf v}_0)$ consists
of all green and red vertices at time $t$, and  $\chs{R_t ({\bf v}_0)=}C_t ({\bf v}_0)
\setminus G_t ({\bf v}_0)$ is the set of red vertices \chs{at time $t$.} All the
remaining vertices in the graph are white.

Let $T_{{\bf v}_0}$ denote the smallest time $t$ when there are no green
vertices remaining, that is $T_{{\bf v}_0}= \inf \{t: |G_t ({\bf v}_0)|=0\}$.
Note that $T_{{\bf v}_0} = |C({\bf v}_0)|$, the size
of the cluster of vertex ${\bf v}_0$, \chs{and $|R_t ({\bf v}_0)|=t$ for all
$t\leq T_{{\bf v}_0}$.} Choose a parameter $\eta = \eta
(\vep,V)$ such that $0 < \eta \ll \vep$ and let
    \eq
    \lbeq{Tdef}
    T = T_{{\bf v}_0} \land \lceil \eta V \rceil
    \en
be the minimum of
$T_{{\bf v}_0}$ and $\lceil \eta V\rceil$.

Given an integer $i \in \{1, \ldots , n\}$, let $C_t ({\bf v}_0,i)$ be the set
of vertices $(i,y)$ included in the
cluster at time $t$. (This is the collection of all the elements of
the $i$-th horizontal line added by time $t$ during the exploration
procedure.) Let also $C({\bf v}_0,i)$ be the set
of all vertices $(i,y)$ in $C ({\bf v}_0)$, that is, the collection of
all the elements in the $i$-th horizontal line contained in $C ({\bf
v}_0)$. We further
denote the number of elements
of the $i$-th horizontal line included in $C({\bf v}_0)$ until time
$T$ by $N({\bf v}_0,i) = |C_{\sT}({\bf v}_0,i)|$.

Similarly, let $\hat{C}_t ({\bf v}_0,i)$ be the set
of vertices $(x,i)$ included in the cluster at time $t$, that is
all the $i$-th vertical line elements added by
time $t$ during the exploration procedure.) Let $\hat{C} ({\bf v}_0,i)$ be the set
of all vertices $(x,i)$ in $C ({\bf v}_0)$; and, finally,
denote the number of elements of the $i$-th vertical line included in
$C ({\bf v}_0)$ until time $T$ by $\hat{N}({\bf v}_0,i) = |\hat{C}_{\sT} ({\bf v}_0,i)|$.

We write $(x_t,y_t)$ for the vertex that is explored at time $t$ if such a
vertex exists, that is, if $t \le T$.
We may identify the set of colours with the set $\{0,1,2\}$.
The state of the exploration process at time $t$ is the list giving
the colour of each vertex, in other words, an $n$-vector with values in
${\{0,1,2\}}^n$. This process
defines a natural filtration $\phi_0 \subseteq
\phi_1 \subseteq \ldots \subseteq \phi_{\sss T}$, where $\phi_t$ is the smallest
$\sigma$-field with respect to which the state at time $t$ is measurable.
(Informally, $\phi_t$ corresponds to ``everything that has
occurred until time $t$''.)
We note that $T$ is a
stopping time with respect to this filtration.
We note also that, even  on the event $\{T = \lceil\eta V\rceil\}$, it is not
necessarily the case that $C_{\sss T} ({\bf v}_0) = \lceil \eta V\rceil$, since
the number of new vertices added at each exploration step is a random
variable, which can be smaller or greater than 1.
We stop our process at time $T$, and we
make the convention that $C_t ({\bf v}_0) = C_{\sss T} ({\bf v}_0)$ for all
$t \ge T$ (and similarly for all other relevant random variables). This
is important when $T=T_{{\bf v}_0} < \eta V$, that is, when the process dies
out before time $\eta V$.

Following the notation of Section~\ref{S: branching}, we let $F$ denote
the total population size of a Galton-Watson
process starting with one individual, where the offspring
distribution is $\Bi(\cn,p)$; and further $\prob_{\cn,p}$ denotes the
probability measure corresponding to this branching process.
Proposition~\ref{lem-CCS}
involves upper and lower bounds on $\Pr (|C({\bf v}_0)|\geq \ell)$
for appropriate choices of $\ell$. These bounds
are formulated in Lemmas~\ref{prop-couplingUB}--
~\ref{prop-couplingLB} below.

\begin{lemma}
[Stochastic domination of cluster size by branching process progeny size]
\label{prop-couplingUB}
For every $\ell \in \bbN$,
    \ch{\eqn{
    \Pr (|C({\bf v}_0)|\geq \ell) \leq
    \prob_{\cn,p}(F\geq \ell).
    }}
\end{lemma}

\ch{The result in Lemma~\ref{prop-couplingUB} is standard, and we
will omit its proof. In essence, it follows since in the cluster
exploration, from each vertex being explored, {\it at most}
$\Bi(\cn, p)$ new vertices can be added to the cluster,
independently of what has already been added. Thus the total cluster
size must be at most the total population size of the binomial
Galton-Watson process, as claimed.}

Lemma~\ref{prop-BP} below follows directly from
Lemma~\ref{prop-couplingUB} \ch{and Proposition~\ref{prop.branch-3},}
and establishes the upper bound part of Proposition~\ref{lem-CCS}.
It is also used in the proof of Lemma~\ref{prop-Zsmall}
in Section~\ref{S: conc}.

\begin{lemma}
\label{prop-BP}
For every $\ell \in \bbN$, and for $\vep\geq V^{-1/3}$,
    \begin{align*}
    \Pr (|C({\bf v}_0)|\ge \ell )& \le 2 \vep + O(\vep^2)+
    O\Big ( \frac{1}{\sqrt{\ell}} \Big ).
    \end{align*}
In particular, if $\vep^{-2} \ll \eta V \ch{\ll} \vep V$, then
    \eqn{
    \Pr (|C({\bf v}_0)| \geq \eta V) \le 2\vep (1+o(1)).
    }
\end{lemma}

\begin{proof}
By Lemma~\ref{prop-couplingUB}, for every $\ell \in \bbN$,
    \begin{align*}
    \Pr (|C({\bf v}_0)|\geq \ell) \leq
    \prob_{\cn,p}(F\geq \ell).
    \end{align*}
Our choice of $p=p(n)$ implies that $\cn p=1+\vep  > 1$, that is, the
$\Bi (\cn,p)$ Galton-Watson process is supercritical.
By Proposition~\ref{prop.branch-3},
    \begin{align*}
    \prob_{\cn,p} (F \ge \chs{\ell})& =2 \vep + O(\vep^2)
    + O\Big ( \frac{1}{\sqrt{\chs{\ell}}} \Big ).
    \end{align*}
\chs{For $\ell=\eta V$, we have that} $\eta V\gg \vep^{-2}$,
and so $1/\sqrt{\eta V}= o(\vep)$,
which completes the proof.
\end{proof}

Our next \ch{lemma} establishes a lower bound on $\Pr (|C({\bf
v}_0)|\geq \ell)$, that is, the lower bound part of
Proposition~\ref{lem-CCS}.
\begin{lemma}
[Stochastic domination of cluster size over branching process progeny size]
\label{prop-couplingLB}
For every $\ell \ll \vep V$,
    \begin{align}
    \label{eqn.lower-1}
    \Pr (|C({\bf v}_0)|\geq \ell) \geq
    \prob_{\chs{\tcn},p}(F\geq \ell)+O(V^{-3}),
    \end{align}
where $\chs{\tcn}= \cn - \frac{5}{2} \max \{\ell n^{-1},C \log n\}$.\\
Consequently, if $\vep\gg V^{-1/3}$, $\eta \ll \vep$ and $\vep^{-2}
\ll \eta V$, then 
    \eqn{
    \label{eqn.lower-2}
    \prob_p(|C({\bf v}_0)|\geq \eta V)\geq 2\vep (1+o(1)).
    }
\end{lemma}

Lemma~\ref{prop-couplingLB} is proved in Section~\ref{S: lower},
where we show that the cluster size stochastically dominates the
$\Bi (\chs{\tcn},p)$ Galton-Watson process.
\smallskip

\subsection{\chs{Upper bounds on the cluster size and structure}}
\label{S: upper}
In this \chs{section} we give an upper bound on the number of elements of
a large cluster that belong to a particular horizontal line.
The following proposition is crucial in the proofs of
Proposition~\ref{lem-goodlines} and \chs{Lemma~\ref{prop-couplingLB}}.

\begin{proposition}[Upper bound on the number of elements per line in
a large cluster]
\label{prop-badlines}
Let \\ $\vep =\vep (n)\geq 0$ be such that $\vep =\vep (n)\le 1/20$
and choose $\eta \ll \vep$. Further, let
$N({\bf v_0},i)$ be the number of elements of $C_{\sT}({\bf v_0})=C_{\sss \chfin{\lceil \eta V \rceil}}
({\bf v_0})$ that belong to the horizontal line $i$. There exists a
positive constant $c_1$ such that
for every $\nu > 0$
    \eqn{\prob_p\Big(\max_{i =1, \ldots , n}: N({\bf v_0},i) \ge (1+\nu)
    \frac{11}{9}\eta n\Big)
    \le ne^{-c_1\nu \eta n}.}
Furthermore, there exist constants $c_2,c_3,c_4 > 0$ such that the following
holds:
\begin{enumerate}
\item  Let $n \in \bbN$ and $\eta = \eta (n)$ be such that $\eta n \ge
c_2 \log n$. If $n$ is sufficiently large, then
    \eqn{\prob_p \Big (\max_{i =1, \ldots , n}: N({\bf v_0},i) \ge
    \frac54 \eta n \Big )\le c_4 V^{-3}.}
\item Let $n \in \bbN$ and $\eta = \eta (n)$ be such that $\eta n/
\log n < c_2$. If $n$ is sufficiently large, then
    \eqn{\prob_p\Big(\max_{i =1, \ldots , n}: N({\bf v_0},i) \ge c_3 \log n\Big)
    \le c_4 V^{-3}.}
\end{enumerate}
\end{proposition}

Here is an informal outline of the proof.
Whenever we explore a vertex {\it not} on the line $i$, we add an
element of line $i$ with probability $p$. On the other hand, each
vertex belonging to the line $i$ has $n-1$ neighbours on that
line. Whenever such a vertex is explored, each one of its neighbours
on the line $i$ is included with probability $p$ (unless it is already
in the cluster). It follows that the number of new elements on line $i$
resulting from exploring a vertex belonging to that line
is stochastically dominated by a $\Bi (n-1,p)$ Galton-Watson process.
Since $p= \frac{1+ \vep}{2(n-1)}$ and $\vep \le 1/2 < 1$ for $n$ large
enough, we have that $(n-1)p <1$, so that the Galton-Watson process
is subcritical. Hence, using standard concentration of measure
techniques, we are able to upper bound the number of elements
on line $i$ that make it into a large cluster. We now make this argument precise.

\smallskip

\begin{proofof}{Proposition~\ref{prop-badlines}}
Let $i\in \{1, \ldots , n\}$ and, for each
$t=1, \ldots,T$, let $S_t(i)$ be the number of times $s$ such that
$(x_{s-1},y_{s-1})$, $(i,y_{s-1})$ is open and $x_{s-1}
\not = i$. 
That is, for each time $t \le T$, $S_t(i)$ is the number of times we
enter the horizontal line $i$ until time $t$.
We can write
$$S_t(i) = \sum_{s=1}^t Y_s(i),$$ where
$Y_t(i)$ is the indicator of the event that the
edge between $(x_{t-1},y_{t-1})$, $(i,y_{t-1})$ is open, and $x_{t-1}
\not = i$. For each time $t$,
    $$\Pr (Y_t(i) =1|\phi_{t-1}) \le p,
    $$
and so known results (see for instance Lemma 2.2 in~\cite{lmu03})
imply that $S_t(i)$ is
stochastically dominated by a $\Bi (t,p)$ \rv. Consequently,
for every $u \ge 0$, the \mgf{} $M_{S_t(i)}(u)$ is
bounded above by $(1+p(e^u-1))^t$.

For $r=1, 2, \ldots$, let
$Z_r(i)$ be the number of vertices $(i,x)$ added as a result of the
$r$-th entry on to horizontal line $i$.
Given that vertex $(i,\tilde{x}_0) \in C ({\bf v}_0)$, the number of
its neighbours $(i,x)$ added to  $C ({\bf v}_0)$ during its
exploration (if it has occurred by the time $\eta V$) is easily seen to be
stochastically dominated by a \rv{} $\Bi
(n-1,p)$.
Hence, for each $r$, $Z_r(i)$
is stochastically dominated by the total progeny in a
branching process with offspring distribution $\Bi (n-1,p)$
descending from a single individual. Since $p = (1+\vep)/2(n-1)$ and
$\vep < 1/2$, this
branching process is subcritical. We deduce that, for $u \ge 0$, the
\mgf{}
$M_{\sss Z_r(i)} (u)$ of $Z_r(i)$ is bounded above by the
\mgf{} $M_{\sss Z}(u)$ of an integer-valued, finite
\rv{} $Z$, whose distribution is given by the Otter-Dwass formula (Lemma \ref{lem-OD}). In other words, for
each $N \in \bbN$,
    \eqn{\label{eq-SD-1}
    \Pr (Z=N) = \frac{\prob (\xi_1 + \ldots + \xi_N = N-1)}{N},
    }
where the $\xi_r$ are \iid{} $\Bi (n-1,p)$.
It follows that
    \begin{align*}
    M_{\sss Z}(u) & = \sum_{N=1}^{\infty} \frac{e^{uN}}{N}\prob \big(\Bi (N(n-1), p ) =
    N-1\big)
    \\ &= \sum_{N=1}^{\infty} \frac{e^{uN}}{N} {{N (n-1)} \choose N-1}
    p^{N-1} (1-p)^{N(n-1)-(N-1)}.
    \end{align*}
Our aim is to derive an upper bound for the above expression. Unlike the
branching processes considered in Section \ref{S: branching}, which
were (slightly) supercritical, we are now
{\it subcritical}. Recall that
the expected total progeny of a $\Bi (m,p)$ Galton-Watson
process is $\frac{1}{1-mp}$; using this fact with $m=n-1$ and
$p=\frac{1+\vep}{2(n-1)}$,
we see that $\expec[Z]=\frac{2}{1-\vep}$.

As $N!\geq (N/e)^N$, we have
    \eqn{
    \frac{1}{N}{{N (n-1)} \choose N-1}\leq \frac{[N (n-1)]^{N-1}}{N!}
    \leq \frac{(n-1)^{N-1}}{N}e^{N},
    }
which in turn implies that
    \begin{align*}
    M_{\sss Z}(u) & \leq \sum_{N=1}^{\infty}
    \frac{e^{uN}}{N}(n-1)^{N-1}e^{N}
    \Big ( \frac{1+\vep}{2(n-1)}\Big )^{N-1} \Big (1-
    \frac{1+\vep}{2(n-1)}\Big )^{N(n-1)-(N-1)}.
    \end{align*}
Since $1-x\leq e^{-x}$, we see that
    \begin{align}
    \lbeq{mgf-bd1}
    M_{\sss Z}(u) &\leq \frac{2}{1+\vep}
    \sum_{N=1}^{\infty}\frac 1{N} \Big ( \frac{e^{u+1}(1+\vep)e^{-(1+\vep)/2}}{2}\Big )^N
    e^{\frac{(1+\vep)(N-1)}{2(n-1)}},
    \end{align}
which is finite for $0 \le u < \frac{1+\ep }{2} -1 -
\log \frac{1+\ep }{2}$ and $n$ large.
Further, it is easily seen that for such $u$ and $n$ the contribution
due to terms with $N > C \log n $ is negligible, provided that $C$ is a
sufficiently large constant.

Clearly, $C_t ({\bf v}_0,i)$
is always bounded above by $\sum_{r=1}^{S_t(i)}Z_r(i)$.
In particular, $\sum_{r=1}^{S_{\sss T}(i)} Z_r(i)$ is an upper bound on $C_{\sss T}
({\bf v}_0,i)$, which equals the
number of vertices $(i,x)$ included in the cluster of ${\bf v}_0$ from
time $0$ to time $\eta V$.

Let $u_0 = \frac12 (\frac{1+\vep}{2} -1 -
\log\frac{1+\vep}{2})$. Since $Z$ is non-negative, $M_{\chs{{\sss Z}}}(u) \ge 1$.
It follows from the above (using also the bound \chs{$1+x \le e^x$}) that for $0 \le u \le u_0$,
    \begin{align}
    \lbeq{mgf-bd2}
    M_{N({\bf v},i)}(u) & \le M_{S_{\eta V}(i)} (\log
    M_{\sss Z}(u))
    \le (1+
    p(M_{\sss Z}(u) -1))^{\lceil \eta V\rceil}
    \\&
    \le \exp \Big (\chs{\lceil \eta V\rceil}
    p(M_{\sss Z}(u)-1)\Big )
    \nonumber\\
    &
    = (1+o(1))\exp \Big (\frac12 \eta n(1+\vep)
    (M_{\sss Z}(u)-1)\Big )
    \nonumber\\&
    \le (1+o(1)) \exp \Big ( \Big [\eta n\frac{1+\vep}{1-\vep}u +
    \frac12 \eta n (1+\vep )u^2 M''_{\sss Z}(u_0) \Big ]\Big ),
    \nonumber
    \end{align}
where the final inequality comes from a second order Taylor
expansion. Also we have used the fact that
$\expec[Z]=2/(1-\vep)$,
and that the second derivative $M''_{\sss Z}(u)$ is
increasing in $u$ for $u \le u_0$.

Now we run a standard large deviations argument.
For all $k$,
    \begin{align}
    \Pr (N({\bf v},i) \ge k) & \le \frac{M_{N({\bf v},i)}(u)}{e^{uk}}
    \nonumber
    \\
    & \le (1+o(1))\exp \Big (\Big (\eta n\frac{1+\vep}{1-\vep}-k\Big )u +
    \frac12 \eta n (1+\vep )u^2 M''_{\sss Z}(u_0)\Big ).
    \label{eq-largedev}
    \end{align}
The expression in~(\ref{eq-largedev}) can be optimised with
respect to $u$ in the usual way, and one finds that there exists a
constant $c_1> 0$ such that for all $\nu > 0$,
    $$\Pr \Big (N({\bf v},i) \ge (1+\nu ) \eta n
    \frac{1+\vep}{1-\vep}\Big ) \le e^{-c_1\nu \eta n},
    $$
which yields
    $$\Pr \Big (\max_{ i =1, \ldots , n} N({\bf v},i) \ge (1+\nu ) \eta n
    \frac{1+\vep}{1-\vep} \Big ) \le n e^{-c_1\nu \eta
    n}.$$
Since $\vep \le 1/20$, we have $(1+\vep)/(1-\vep) \le 11/9$; and
therefore there exists a constant $\tilde{c}_1$ such that, for $\nu  >0$,
    \eqn{\label{eq-largedev-1}
    \Pr \big(\max_{ i =1, \ldots , n} N({\bf v},i) \ge \frac{11}{9} (1+\nu ) \eta n\big)
    \le n e^{-2\tilde{c}_1\nu \eta n},}
which proves the first statement of Proposition~\ref{prop-badlines}.

For the remainder, first suppose that $\eta n/\log n \ge 176
/\tilde{c}_1$; we may
take $\nu =1/44$ in~(\ref{eq-largedev-1}) to deduce that
    $$
    \Pr \Big(\max_{ i =1, \ldots , n} N({\bf v},i) \ge \frac54 \eta n\Big)
    \le n^{-6}=V^{-3}.
    $$
Now assume that $\eta n/\log n < 176 /\tilde{c}_1$. Note that
$N({\bf v},i)$ is stochastically dominated by $\sum_{r=1}^{\tilde{S}}
Z_r(i)$, \chs{where $\tilde{S}=\Bi (\lceil \eta V\rceil,p)$.
Since $\eta V<176 \log{n} n/\tilde{c}_1$, $N({\bf v},i)$ is stochastically
dominated by a $\Bi (\lceil 176\tilde{c}_1^{-1}n \log n\rceil,p)$ random variable.}
We can perform the moment
generating function and large deviations calculations as in \refeq{mgf-bd1}--
\eqref{eq-largedev} above, to
find that
    $$
    \Pr \Big (\max_{ i =1, \ldots , n} N({\bf v},i) \ge
    \frac{220}{\tilde{c}_1} \log n \Big ) \le n e^{-8 \log n(1+o(1))}=o(V^{-3}).
    $$
Taking $c_2=176/\tilde{c}_1$, $c_3 = 220/\tilde{c}_1$ and
$c_4=1$ completes the proof of
Proposition~\ref{prop-badlines}.
\end{proofof}

\subsection{\chs{Lower bounds on the cluster size and structure}}
\label{S: lower}

In this section we establish corresponding lower bounds on the
cluster size and structure. We first give a proof of
Lemma~\ref{prop-couplingLB}, which will in particular establish a
lower bound on $\Pr (T =\lceil \eta V\rceil)$, that is, a lower
bound on $\Pr (|C({\bf v}_0)| \ge \eta V)$. Our argument will rely
on a coupling with a suitable lower bounding Galton-Watson process
and the estimates of Proposition \ref{prop-badlines}.

\medskip

\begin{proofof}{Lemma~\ref{prop-couplingLB}}
As in Sections~\ref{S: explore} and~\ref{S: upper}, $C_t({\bf
v}_0,i)$ denotes the set of vertices $(i,y)$ \ch{ (with $i\in \{1,
\ldots, n\}$ is fixed and $y\in \{1,\ldots, n\}$ varying)} included
in the exploration of cluster $C({\bf v}_0)$ until time $t$; and
$\hat{C}_t({\bf v}_0,i)$ denotes the set of vertices $(x,i)$
included until time $t$.

Let $c_1$ be as in Proposition~\ref{prop-badlines}. Let $m = \frac54
\max \{\eta n,\frac{176}{c_1}\log n\}$. For each time $t$, let
${\mathcal E}_t$ be the event that $|C_t({\bf v}_0,i)| \le m$ and
$|\hat{C}_t({\bf v}_0,i)| \le m$ for all $i=1, \ldots, n$.

Define $\chs{\tcn} = \cn -2m = 2(n-1-m)$.
Then, provided that the event ${\mathcal E}_t$ occurs, conditionally on
$\phi_t$ (that is, given everything else that may have happened until
time $t$), the number of
vertices added to $C_t({\bf v}_0)$ as a result of exploring $(x_t,y_t)$
stochastically dominates a $\Bi (\chs{\tcn}, p)$ \rv.
We note that $\cn - \chs{\tcn} = O(\eta n + \log n) = o(n)$, since $\eta \to
0$ as $n \to \infty$.

We shall couple our exploration process with a
Galton-Watson process starting with a single individual, where the
offspring distribution is $\Bi (\chs{\tcn}, p)$. The mean offspring
size for this Galton-Watson process is
    $$\chs{\tcn}p \= 1+ \chs{\tvep} =\Big (1-\frac{m}{n-1} \Big )(1+\vep )
    = 1 +
    \vep -O(\eta + n^{-1} \log n )= 1+ \vep (1+o(1)),
    $$
where we have used the fact that $\eta \ll \vep$ and $n^{-2/3} \ll \vep$.
By Proposition~\ref{prop.branch-3}, its survival
probability is $2 \vep + O(\eta + n^{-1}\log n+ \vep^2) = 2\vep(1+o(1))$.

\ch{Recall the exploration process and its corresponding colours as
described in Section \ref{S: explore}.}
Let $F_t$ be the population size \ch{of the $\Bi (\chs{\tcn},p)$
Galton-Watson process} and let $F'_t$ be the set of {\em green} or
active individuals \ch{in the Galton-Watson process} at time $t$.
Also, let $F= \sup_t F_t$ be the total population size \ch{of the
Galton-Watson process. Finally, recall that $C_t({\bf v}_0)$ is the
set of red and green vertices in the exploration of the cluster of
${\bf v}_0$, and $G_t({\bf v}_0)$ the set of green
or active vertices in the cluster exploration. By construction,
$C_t({\bf v}_0)\subseteq C({\bf v}_0)$ for every $t\geq0$.}

By the above, on the event ${\mathcal E}_t$ intersected with the event
that $|C_t({\bf v}_0)| \ge F_t$ and $|G_t({\bf v}_0)| \ge
F'_t$, given \chs{$\phi_t$}, we can couple the Galton-Watson process with
the cluster exploration processes
for another step so that $|C_{t+1}({\bf v}_0)| \ge F_{t+1}$ and
$|G_{t+1}({\bf v}_0)|\ge F'_{t+1}$.

It follows by induction that for each $t$
the \rv{} $|C_t({\bf v}_0)|I[{\mathcal E}_t]$ is
stochastically at least $F_tI[{\mathcal E}_t]$. Hence,
for each $k$,
    \begin{align*}
    \Pr ( |C_t({\bf v}_0)| \ge k) & \ge
    \Pr \big({\mathcal E}_t  \cap \{|C_t({\bf v}_0)| \ge k\}\big)
    \ge \prob_{\cn,\chs{\tcn},p} \big({\mathcal E}_t \cap \{F_t \ge k\}\big)
    \\& \ge  \prob_{\chs{\tcn},p}
    (F_t \ge k) - \Pr (\mathcal E_t^c),
    \end{align*}
where $\prob_{\cn,\chs{\tcn},p}$ denotes the coupling measure. In the second
inequality, we have used the fact that
for every pair of events ${\mathcal A},{\mathcal B}$, we have $\Pr
({\mathcal A} \cap {\cB}) \ge \Pr ({\mathcal A}) -\Pr (\cB^c)$.

By Proposition~\ref{prop-badlines}, $ \Pr (\ch{\mathcal
E^c_t})=O(V^{-3})$, and so, for each $t \le \eta V$, we obtain
    $$
    \Pr ( |C_t({\bf v}_0)| \ge k)
    \ge   \prob_{\chs{\tcn},p} (F_t \ge k) + O(V^{-3}),
    $$
which establishes~\eqref{eqn.lower-1}. Similarly, for each time $t$
and non-negative integer $k$,
    $$
    \Pr ( |G_t({\bf v}_0)| \ge k) \ge \prob_{\chs{\tcn},p}
    (F'_t \ge k) - \Pr (\mathcal E_t^c),
    $$
and, in particular, for each $t \le \eta V$,
    $$
    \Pr (T \ge t) = \Pr ( |G_t({\bf v}_0)| \ge 1) \ge \prob_{\chs{\tcn},p}
    (F'_t \ge 1) + O(V^{-3}).
    $$
Notice that, for $t \le \eta V$,
    \begin{align*}
    \prob_{\chs{\tcn},p}(F'_t =0)
    & \le  \prob_{\chs{\tcn},p} (F'_{\chs{\lceil \eta V\rceil}} =0)
    \le \prob_{\chs{\tcn},p} (F <\infty).
    \end{align*}
In this way we arrive at
    \begin{align*}
    \Pr ( |C_t({\bf v}_0)| \ge k) & \ge \prob_{\chs{\tcn},p} (F_t \ge k) + O(V^{-3})\\
    & \ge \prob_{\chs{\tcn},p}(F_t \ge k, F'_{t} > 0)+ O(V^{-3})\\
    & = \prob_{\chs{\tcn},p} (F'_{t} > 0) - \prob_{\chs{\tcn},p} (F'_t>0,F_t < k) +
    O(V^{-3})\\
    & \ge \prob_{\chs{\tcn},p} (F'_{t} > 0) - \Pr (\Bi
    (t\chs{\tcn},p)< k) + O(V^{-3})\\
    & \ge \prob_{\chs{\tcn},p} (F =\infty) - \Pr (\Bi
    (t\chs{\tcn},p)< k) + O(V^{-3}),
    \end{align*}
since on the event that the process is alive at time $t$ and the event
${\mathcal E}_t$ occurs,
we can couple the number of vertices added at all steps until $t$ so
that it is at least as large as
a sum of $t$ independent binomials $\Bi (\chs{\tcn},p)$.

Hence,  \chs{also using the facts that $p\geq 1/\cn$ and
$|\cn-\chs{\tcn}|=o(\cn)$, for every constant $\gd \in (0,1)$ there
is a constant $\alpha > 0$ such that}
    \begin{align}
    \Pr \Big( |C({\bf v}_0)| \ge (1-\gd ) \eta V\Big) & \ge
    \Pr ( |C_{\eta V} ({\bf v}_0)| \ge (1-\gd ) \eta V) \label{eq-later}
    \\
    & \ge \prob_{\chs{\tcn},p}(F =\infty)
    - \Pr \big(\Bi (\eta V \chs{\tcn},p) < (1-\gd) \eta V\big)
    +O(V^{-3}) \nonumber
    \\
    & = 2\vep +O(\eta+n^{-1}\log n + \vep^2) +e^{-\alpha \eta
    V}+O(V^{-3}). \nonumber
    \end{align}
But equally, we could run the exploration process until time
$(1+\gd)\eta V$ to obtain a cluster of size $\eta V$ \whp, that is,
we could use the above with $\eta$ replaced by $\eta/(1-\gd)$
to obtain that
    \eqn{
    \Pr ( |C({\bf v}_0)| \ge \eta V)\ge 2\vep
    +O(\eta + n^{-1}\log n + \vep^2) + e^{-\alpha \eta V}+O(V^{-3}).
    }
This establishes~\eqref{eqn.lower-2}, thus completing the proof of
\chs{Lemma}~\ref{prop-couplingLB}, and hence also the proof of
Proposition~\ref{lem-CCS}.
\end{proofof}

\smallskip

Let us call a horizontal line {\it good}
if it contains at least $\eta V/(4n)=\eta n/4$ elements in
$C ({\bf v_0})$ along that line, and {\it bad} otherwise.
We shall now prove Proposition~\ref{lem-goodlines}, thus establishing
a lower bound on the number of good lines.
\medskip

\begin{proofof}{Proposition~\ref{lem-goodlines}}
As earlier, for a vertex
${\bf v}_0$ and $i \in \{1, \ldots , n\}$,
the random variable $C_t({\bf v}_0,i)$
denotes the number of elements of the
$i$-th horizontal line contained in $C_t({\bf v}_0)$,
the part of $C({\bf v}_0)$ obtained by running the exploration
process until time $t$. Also, $C({\bf v}_0,i)$ is the number
of elements of the $i$-th horizontal line in $C({\bf v}_0)$ and $N({\bf
v}_0,i)$ is the number of such elements included in $C({\bf v}_0)$ until
time $\lceil \eta V\rceil$.

Let $c_2$ be as in Proposition~\ref{prop-badlines}, \chfin{statement (1),} and choose
$C=c_2$. By Proposition~\ref{prop-badlines}, statement~(1), we have
    $$
    \Pr \Big(\max_{{\bf v}_0} \max_{i =1, \ldots , n}: N({\bf v}_0,i) \le
    \frac54 \eta n\Big) =O(V^{-3}).
    $$

\chs{We select a vertex ${\bf v}_0$. Let ${\mathcal A}$ be the event
that $\{\max_{i =1, \ldots , n} N({\bf v}_0,i) \le \frac54 \eta
n\}$. Let also ${\mathcal B}$ (``${\mathcal B}$'' for ``bad'') be
the event that fewer than $3n/4$ lines are good for the cluster
$C({\bf v}_0)$. On the event that $|C({\bf v}_0)|\geq \eta V$, we
have $|C_{\sss \lceil \eta V\rceil}({\bf v}_0)|\geq |R_{\sss \lceil
\eta V\rceil}({\bf v}_0)| =\lceil \eta V\rceil$. It follows that we
only need to show that
    \eqn{\label{eq-one-good}
    \P_p\big({\mathcal B} \cap \{|C({\bf v}_0)|\geq \eta V\}\big)
    =\P_p\big({\mathcal B} \cap \{|C_{\sss \lceil \eta V\rceil}({\bf v}_0)|\geq \eta V\}\big)
    =o(V^{-1}).
    }

Indeed, summing over all vertices ${\bf v}_0$ we may deduce
from~(\ref{eq-one-good}) that \whp{} there is no ${\bf v}_0$ such
that $|C({\bf v}_0)|\geq \eta V$ and fewer than $3n/4$ lines are
good for $C({\bf v}_0)$. In order to establish~(\ref{eq-one-good}),
we shall show that
    \eqn{\label{claim:line}
    \Pr({\mathcal B} \cap
    \{|C_{\sss \lceil \eta V\rceil}({\bf v}_0)|\geq \eta V\}) \le \Pr({\mathcal A}^c).
    }
Clearly, $|C({\bf v}_0,i)| \ge N({\bf v}_0,i)$ for every $i$. Let us write
$g_{{\bf v}_0}$ and
$b_{{\bf v}_0}$ respectively for the number of good and bad lines in
$C_{\sss \ch{\lceil \eta V\rceil}} ({\bf v}_0)$.

On the event ${\mathcal A}$, the explored cluster
 $C_{\sss \lceil \eta V\rceil}({\bf v}_0)$ at time $\eta V$ contains
at most $5\eta n/4$ elements of every good line and at the same time
has size at least $\eta V$. Hence, using also that $g_{{\bf
v}}=n-b_{{\bf v}_0}$, on ${\mathcal A} \cap \{|C_{\sss \lceil \eta
V\rceil}({\bf v}_0)|\geq \eta V\}$,
    \begin{align*}
    \eta V & \leq |C_{\sss \eta V} ({\bf v}_0)|
    \leq \frac54 \eta n g_{{\bf v}_0} + \frac{1}{4}\eta n b_{{\bf
    v}_0}
    =
    \eta n g_{{\bf v}_0}+ \frac14 \eta V,
    \end{align*}
which gives
    \begin{align*}
    \frac{3}{4} \eta V  & \le \eta n g_{{\bf v}_0}
    \end{align*}
and hence
    \begin{align*}
    g_{{\bf v}_0} & \ge \frac34 n.
    \end{align*}
In other words, on ${\mathcal A} \cap \{|C_{\sss \lceil \eta V\rceil}({\bf v}_0)|\geq \eta V\}$,
the number of good lines is at least $3n/4$, which
means that
    \eqn{
    \Pr ({\mathcal B} \cap {\mathcal A} \cap
    \{|C_{\sss \lceil \eta V\rceil}({\bf v}_0)|\geq \eta V\})=0,
    }
and so establishes claim~(\ref{claim:line}). Then, from
Proposition~\ref{prop-badlines}, we see that
    \eqn{
    \lbeq{prop5.1-res}
    \Pr ({\mathcal B} \cap \{|C_{\sss \lceil \eta V\rceil}({\bf v}_0)|\geq \eta V\})\leq\Pr
    ({\mathcal A}^c) = O(V^{-3}),
    }
as required. This completes the proof of Proposition~\ref{lem-goodlines}.}
\end{proofof}

\medskip

\section{Concentration of measure for the number of vertices in large
clusters} \label{S: conc}

\ch{This section contains our proof of Proposition~\ref{lem-varZ}.
In outline, the goal is to establish concentration of measure for
$Z_{\sss \geq \alpha^{-2}}$, for an appropriate choice of $\alpha\ll
\vep$ to be determined below. This will be carried out by second
moment methods, in a slightly unusual way, as we explain now.

For every $\ell$, define the centered versions of the
random variables $Z_{\sss \geq \ell}$ by
    \eqn{
    \lbeq{Zbar-def}
    \Zbar_{\sss \geq \ell}
    =Z_{\sss \geq \ell}-
    \expec_p[Z_{\sss \geq \ell}].}
The entire proof revolves around two scales of magnitude, denoted by
$\Nul$ and $\NV$. The value $\NV$ is the large scale, and
corresponds to $\vepnull^{-2}$ in Proposition~\ref{lem-varZ}. The
value $\Nul$ is the smaller scale, and will be determined below. The
scales $\Nul$ and $\NV$ are related through a positive integer $I$
defined by
    \eqn{
    \lbeq{NV-choice}
    \NV = \Nul 2^I.
    }
With this notation, proving Proposition~\ref{lem-varZ} amounts to
establishing that
    \eqn{
    \prob_p\Big(|\Zbar_{\sss \geq \NV}|\geq
    \delta \vep V\Big)=o(1).
    }
We first observe that
    $$|\Zbar_{\sss \geq \NV}| \le |\Zbar_{\sss \geq \Nul}| +
    \sum_{i=0}^{I-1} |\Zbar_{\sss
    \geq 2^{i+1}\Nul}- \Zbar_{\sss \geq 2^{i}\Nul}|.
    $$
The goal is now to establish sufficient bounds on the variances of
the above random variables, so that we can prove that $\Zbar_{\sss
\geq \NV}$ is concentrated. For this, we choose a sequence
$\{\delta_i\}_{i=0}^{I-1}$ such that each $\delta_i>0$ and
$\sum_{i=0}^{I-1} \delta_i \leq \frac{\delta}{2}$. If $|\Zbar_{\sss
\geq \NV}|\geq \delta \vep V$, then either $|\Zbar_{\sss \geq
\Nul}|\geq \delta \vep V/2$, or $|\Zbar_{\sss \geq 2^{i+1}\Nul}-
\Zbar_{\sss \geq 2^{i}\Nul}|\geq \delta_i \vep V$ for some $0\leq
i\leq I-1$. Consequently,
    \eqn{
    \lbeq{decomp-conc}
    \prob_p\Big(|\Zbar_{\sss \geq \NV}|\geq
    \delta \vep V\Big)
    \leq \prob_p\big(|\Zbar_{\sss \geq \Nul}|\geq
    \delta \vep V/2\big)
    +\sum_{i=0}^{I-1}
    \prob_p\Big(|\Zbar_{\sss \geq 2^{i+1}\Nul}-
    \Zbar_{\sss \geq 2^{i}\Nul}|\geq
    \delta_i \vep V\Big),
    }
and we are going to upper bound each term on the right hand side
separately. Our argument relies on estimating the variance of
$Z_{\sss \geq \Nul}$ and those of the differences $Z_{\sss \geq
2^{i+1}\Nul}-Z_{\sss \geq 2^{i}\Nul}$. This is accomplished in
Section \ref{sec-varest} -- see Lemmas~\ref{prop-Zsmall} and  Lemma
\ref{prop-Zinbetween}. The variance estimates impose various
restrictions on $\Nul$ and $\NV$; in Section \ref{sec-PfvarZ} we
show that these can be satisfied as long as $\vep^3 V\gg \log{n}$,
which establishes Proposition \ref{lem-varZ}. The key to the proof
is to choose $\Nul$, $\NV$ and $\{\delta_i\}_{i=0}^{I-1}$ so as to
ensure adequate concentration of measure.}

\ch{The remainder of this section is organised as follows. In
Section \ref{sec-clustertails} we bound the cluster tail bounds of
the form $\Pr (|C({\bf v}_0)| \in [\ell ,2\ell])$; these are needed
to estimate the distribution of the random variables $Z_{\sss \geq
2^{i+1}\Nul}-Z_{\sss \geq 2^{i}\Nul}$. Here we shall make use of
Galton-Watson processes estimates and comparisons established in
Sections \ref{S: branching}--\ref{S: comparison}. Then, in
Section~\ref{sec-varest}, we upper bound the variances of $Z_{\sss
\geq \Nul}$ and $Z_{\sss \geq 2^{i+1}\Nul}-Z_{\sss \geq 2^{i}\Nul}$.
Finally, in Section \ref{sec-PfvarZ} we complete our proof of
Proposition~\ref{lem-varZ}.}

\subsection{Key ingredients}
\label{sec-clustertails}
As before, for a positive integer $N$ and an edge
probability $p$, $\prob_{N,p}$ denotes the probability
measure corresponding to the Galton-Watson process where the family size is
a $\Bi (N,p)$ random variable; also, $F$ is the total
progeny.

The remainder of this \chs{section} is devoted to establishing a
bound on the cluster tail crucial to the arguments in
Sections~\ref{sec-varest} and  Section \ref{sec-PfvarZ}. Recall that
$\cn=2(n-1)$, \ch{choose a positive integer $\ell$}, and suppose
that $\chs{\tcn}=\chs{\tcn}(n)$ satisfies $\cn - \chs{\tcn}
=\ch{O(\log n+\ell/n)}$ \chs{for some $\ell=o(\vep V)$.} Suppose
further that $\vep = \vep (n)=\cn p -1 \to 0$ such that $V^{-1/3}
\ll \vep \ll 1$ for $V=n^2$ sufficiently large. \ch{Then
$|\vep-\tvep|=p|\cn-\chs{\tcn}|=O(\log{n}/n+\ell/n^2)=o(\vep)$,
\chs{since} $\ell=o(\vep V)$, and so we may use the results of
Proposition~\ref{prop.branch-1}. Hence, as long as $\ell=o(\vep V)$,
we have, uniformly in $n$,}
    \eqn{
    \label{ingre1-diffBP}
    |\prob_{\cn, p}(F\geq \ell)-\prob_{\chs{\tcn}, p}(F \geq \ell)|
    \leq C \Big(p|\cn-\chs{\tcn}|+\frac{1}{n\ell^{1/2}}+\frac{1}{\ell^3}\Big).
    }
We shall use inequality~(\ref{ingre1-diffBP}) in the following lemma
to identify the cluster tail distribution more precisely.

\begin{lemma}[Bound on the cluster tail]
\label{lem.tail}
Set $p=p_c +\frac{\vep}{\cn}$.
Let $\chs{V^{-1/3}\ll \vep \ll 1}$, and
let $\ell \in \bbN$ satisfy
$\ell \le V^{2/3}$ and \ch{$\ell \ll \vep V$.}
Then there exists a constant $C$ such that, for $n$ sufficiently large,
    \eqn{
    \lbeq{ingre2-clustail}
    \Pr (|C({\bf v}_0)| \in [\ell ,2\ell])\leq \frac{C}{\sqrt{\ell}}.
    }
\end{lemma}
\begin{prf}
By \chs{Lemma}~\ref{prop-couplingUB}, 
    \begin{align*}
    \Pr (|C({\bf v}_0)| \ge \ell ) & \le \prob_{\cn,p} (F \ge \ell).
    \end{align*}
Further, by \chs{Lemma}~\ref{prop-couplingLB}, $C$ can
be chosen large enough that
    \begin{align*}
    \Pr (|C({\bf v}_0)| \ge 2\ell ) & \ge \prob_{\chs{\tcn},p} (F \ge 2\ell)+O(V^{-3}),
    \end{align*}
where $\chs{\tcn} = \cn -\frac52 \max \{\ell n^{-1},C \log n\}$.
Let $\chs{\tvep} = \chs{\tcn}p-1$ and note that $|\vep -\chs{\tvep}| \le \frac{C}{n^2}
(\ell + n\log n)$ (after a suitable adjustment of $C$).
Since $\vep\geq V^{-1/3}=n^{-2/3}$, our assumptions
on $\ell$ imply that $|\vep-\chs{\tvep}| =o(\vep)$.
By Propositions~\ref{prop.branch-1} and~\ref{prop.branch-2},
    \begin{align*}
    \prob_{\chs{\tcn},p} (F \ge 2\ell) & \ge \prob_{\cn,p} (F \ge 2\ell) +
    O\big(|\vep -\chs{\tvep}|+\frac{1}{n\ell^{1/2}}+\frac{1}{\ell^3}\big)\\
    & \ge \prob_{\cn,p} (F \ge \ell) + O(|\vep-\chs{\tvep}|) + O(\ell^{-1/2}).
    \end{align*}
 It follows that
    \begin{align*}
    \Pr (|C({\bf v}_0)| \ge \ell ) - \Pr (|C({\bf v}_0)| \ge 2\ell )
    \le O(|\vep-\chs{\tvep}| ) + O(\ell^{-1/2}),
    \end{align*}
and we only need to show that $|\vep-\chs{\tvep}|$ is
$O(l^{-1/2})$. This is equivalent to showing that both
$\frac{\ell}{n^2}$ and $\frac{\log n}{n}$ are $O(l^{-1/2})$.
The condition $\ell\leq V^{2/3}$ is equivalent to
$\frac{\ell}{n^2}\leq \ell^{-1/2}$. The bound
$\frac{\log{n}}{n}\leq \ell^{-1/2}$ holds
when $\ell\leq n^2/(\log{n})^2$; as $\ell\leq V^{2/3}=n^{4/3}$,
this is also true for $n$ sufficiently large.
\end{prf}

\subsection{Variance estimates}
\label{sec-varest}

Lemmas \ref{prop-Zsmall} and
\ref{prop-Zinbetween} below contain
variance estimates essential to our proof of
Proposition~\ref{lem-varZ}.

\begin{lemma}[Variance of the number of vertices in moderate clusters]
\label{prop-Zsmall}
Set $p=p_c +\frac{\vep}{\cn}$. Suppose that \ch{$V^{-1/3}\ll \vep\ll 1$.}
Choose $N$ such that $N=o(\sqrt{V})$ and $N=o(\vep^2 V)$. Then
    \eq
    \var_p(Z_{\sss \geq N})=o\big((\vep V)^2\big).
    \en
\end{lemma}

\begin{prf}
First note that $\var_p(Z_{\sss \ge N})
=\var_p(Z_{\sss <N}),$ where
    $$Z_{\sss <N}=V-Z_{\sss \ge N}
    =\sum_{{\bf v}} I[|C({\bf v})| < N].
    $$
We expand $\var_p(Z_{\sss <N})$ as
    \eqn{
    \var_p(Z_{\sss <N})
    =\sum_{{\bf v}_0,{\bf v}_1} \big[\Pr \big( |C({\bf v}_0)| < N, |C({\bf v}_1)| < N\big)-
    \Pr ( |C({\bf v}_0)| < N)^2\big].
    }
We separate each term involving distinct ${\bf v_0}$ and ${\bf v_1}$
into two, according to whether or not ${\bf v}_1\in  C({\bf v}_0)$. We
can then write
    \eqan{
    \var_p(Z_{\sss <N})
    &=S_{\sss {\bf v}_0\conns {\bf v}_1}+S_{\sss {\bf v}_0\ncs {\bf v}_1},
    }
where $S_{\sss {\bf v}_0\conns {\bf v}_1}=S_{\sss {\bf v}_0\conns
{\bf v}_1}(N)$, $S_{\sss {\bf v}_0\ncs {\bf v}_1}= S_{\sss {\bf v}_0\ncs {\bf
v}_1} (N)$ and
    \eqan{
    S_{\sss {\bf v}_0\conns {\bf v}_1}&=\sum_{{\bf v}_0,{\bf v}_1}
    \Pr ( |C({\bf v}_0)| < N, {\bf v}_1\in  C({\bf v}_0)),\\
    S_{\sss {\bf v}_0\ncs {\bf v}_1}&=\sum_{{\bf v}_0,{\bf v}_1} \Big [
    \Pr ( |C({\bf v}_0)| < N,
    |C({\bf v}_1)| < N, {\bf v}_1 \not \in  C({\bf v}_0))
    -\Pr (|C({\bf v}_0)| < N)^2 \Big ].
    }
It is easily seen that
    \eqn{
    S_{\sss {\bf v}_0\conns {\bf v}_1}
    =V\expec_p\big[|C({\bf v}_0)| I[|C({\bf v}_0)|<N]\big],
    }
and we upper bound
    \eqan{
    \expec_p\big[|C({\bf v}_0)| I[|C({\bf v}_0)|<N]\big]
    &=\sum_{l=1}^N \prob_p(l\leq |C({\bf v}_0)|<N)
    \leq \sum_{l=1}^N \prob_p(|C({\bf v}_0)|\geq l)\nonumber\\
    &\leq C\sum_{l=1}^N \big(\vep +\frac{1}{\sqrt{l}}\big),
    }
where the last inequality follows from Lemma~\ref{prop-BP}. It follows
that
    \eqn{
    S_{\sss {\bf v}_0\conns {\bf v}_1}
    =O(VN\vep +V\sqrt{N})=o(\vep^2V^2),
    }
provided $N=o(\vep V)$ and $N=o(\vep^4 V^2)$. When $\vep^3 V \gg 1$
then $\vep V \ll \vep^4 V^2$, so only the first constraint on $N$ is
binding, i.e. $S_{\sss {\bf v}_0\conns {\bf v}_1}=o(\vep^2V^2)$ as
long as $N=o(\vep V)$.

To upper bound $S_{\sss {\bf v}_0\ncs {\bf v}_1}$ note that,
by \cite[inequality~(9.7)]{bchss1},
    \eqan{
    S_{\sss {\bf v}_0\ncs {\bf v}_1}&\leq p\sum_{\{{\bf u}, {\bf v}\}}
    \expec_p\Big[|C({\bf u})| |C({\bf v})| I[|C({\bf u})|<N, |C({\bf v})|<N,
    {\bf v} \not \in  C({\bf u})]\Big],
    }
where the summation is over all edges $\{{\bf u},{\bf v}\}$ of $H(2,n)$.
We can estimate this similarly to $S_{\sss {\bf
v}_0\conns {\bf v}_1}$ above, and find that
    \eqan{
    S_{\sss {\bf v}_0\ncs {\bf v}_1}&\leq p\sum_{({\bf u}, {\bf v})}
    \sum_{l_1, l_2=1}^N
    \prob_p\Big(l_1\leq |C({\bf u})|<N, l_2\leq |C({\bf v})|<N,
    {\bf v} \not \in  C({\bf u})\Big)\nonumber\\
    &\leq p\sum_{({\bf u}, {\bf v})}
    \sum_{l_1, l_2=1}^N
    \prob_p\Big(|C({\bf u})|\geq l_1, |C({\bf v})|\geq l_2,
    {\bf v} \not \in  C({\bf u})\Big).
    }
Since ${\bf v} \not \in  C({\bf u})$, $|C({\bf u})|$ and $|C({\bf
v})|$ are each {\em independently} of one another stochastically
dominated by the total progeny of a $\Bi (\cn,p)$ Galton-Watson
process. (To see this in more detail, think of first constructing
the cluster of ${\bf u}$, and subsequently construct the cluster of
${\bf v}$ in the smaller graph with $C({\bf u})$ removed.) Using
Lemma~\ref{prop-BP}, we then see that, since $\cn p$ is bounded
above as $n \to \infty$,
    \eqan{
    S_{\sss {\bf v}_0\ncs {\bf v}_1}&\leq C\cn p V
    \sum_{l_1, l_2=1}^N
    \Big(\vep +\frac{1}{\sqrt{l_1}}\Big)\Big(\vep +\frac{1}{\sqrt{l_2}}\Big)\nonumber\\
    &\leq C V  \big(\vep N +\sqrt{N}\big)^2
    \leq O( V \vep^2 N^2+ V N).
    }
Thus, as long as $N=o(\sqrt{V})$ and $N=o(\vep^2 V)$,
    \eqn{
    S_{\sss {\bf v}_0\ncs {\bf v}_1}
    =o\big((\vep V)^2\big),
    }
which completes the proof.
\end{prf}

\smallskip

\noindent
\begin{lemma}
[Variance of the number of vertices in intermediate clusters]
\label{prop-Zinbetween}
Set $p=p_c +\frac{\vep}{\cn}$.
Assume that \ch{$V^{-1/3}\ll \vep\ll 1$,} and take $N \in
\bbN$ such that  $N\leq V^{2/3}$ and
\ch{$N\ll \vep V$.}
Then
    \eq
    \lbeq{VarZbd}
    \var_p(Z_{\sss \geq N}-Z_{\sss \geq 2N})\leq
    \frac{C V^2}{\sqrt{N}}\Big(\frac{\log{n}}{n}\bigvee
    \frac{N}{V}\bigvee \frac{1}{N^3}\Big).
    \en
\end{lemma}

\begin{prf} We have
    $$
    \var_p(Z_{\sss \geq N}-Z_{\sss \geq 2N})
    =\sum_{{\bf v}_0,{\bf v}_1} \Pr ( |C({\bf v}_0)| \in (N,2N],
     |C({\bf v}_1)|\in (N,2N])-\Pr (|C({\bf v}_0)| \in (N,2N])^2.
    $$
Once again we split the sum according to whether or not ${\bf
v}_1\in  C({\bf v}_0)$, obtaining
    \eqan{
    \var_p(Z_{\sss \geq N}-Z_{\sss \geq 2N})
    &=S_{\sss {\bf v}_0\conns {\bf v}_1}+S_{\sss {\bf v}_0\ncs {\bf v}_1},
    }
where \chs{now}
    \eqan{
    S_{\sss {\bf v}_0\conns {\bf v}_1}&=\sum_{{\bf v}_0,{\bf v}_1}
    \Pr \Big( |C({\bf v}_0)|\in (N,2N], {\bf v}_1\in  C({\bf v}_0)\Big),\\
    S_{\sss {\bf v}_0\ncs {\bf v}_1}&=\sum_{{\bf v}_0,{\bf v}_1}
    \Big [\Pr \Big( |C({\bf v}_0)|, |C({\bf v}_1)|\in (N,2N], {\bf v}_1 \not \in  C({\bf v}_0)\Big)
    -\Pr \Big(|C({\bf v}_0)| \in (N,2N]\Big)^2 \Big].
    }
Just as in the proof of Lemma~\ref{prop-Zsmall},
    \eqn{
    S_{\sss {\bf v}_0\conns {\bf v}_1}
    =V\expec_p\big[|C({\bf v}_0)| I[\chs{N<|C({\bf v}_0)|\leq 2N}]\big]
    \leq CV\sqrt{N}.
    }
But
    \eqn{
    V\sqrt{N}\leq \frac{V^2}{\sqrt{N}}\Big(\frac{\log{n}}{n}\bigvee
    \frac{N}{V}\Big),
    }
and so $ S_{\sss {\bf v}_0\conns {\bf v}_1}$ is bounded by the
right hand side of \refeq{VarZbd}.

Dealing with $S_{\sss {\bf v}_0\ncs {\bf v}_1}$ requires more
effort. Define
    \eqn{
    p_{{\bf v}_0, {\bf v}_1}\=\Pr \Big( |C({\bf v}_0)|\in (N,2N],
    |C({\bf v}_1)|\in (N,2N], {\bf v}_1 \not \in  C({\bf v}_0)\Big)
    -\Pr \big(|C({\bf v}_0)| \in (N,2N]\big)^2,
    }
so that
    $$S_{\sss {\bf v}_0\ncs {\bf v}_1}=\sum_{{\bf v}_0,{\bf v}_1}
    p_{{\bf v}_0, {\bf v}_1}.$$
Now rewrite
    \begin{align*}
    \frac{p_{{\bf v}_0, {\bf v}_1}}{\Pr (|C({\bf v}_0)| \in (N,2N])}
    & =
        \Pr \Big(|C({\bf v}_1)|\in (N,2N], {\bf v}_1 \not \in  C({\bf
    v}_0) \big | |C({\bf v}_0)| \in (N,2N]\Big) \\
    &\qquad -\Pr \big(|C({\bf v}_1)| \in (N,2N]\big).
    \end{align*}

Recall that $N({\bf v},i)$ is the number of elements in the $i$-th
horizontal line included in the cluster until time $\eta n^2$.
The proof of Proposition~\ref{prop-badlines} implies that there is some constant $C > 0$
such that \wvhp{}
every ${\bf v}$ such that $|C({\bf v})|\in (N,2N]$ satisfies
    \eqn{
    N ({\bf v},i)\leq C \Big [\log{n}\vee \frac{N}{n} \Big].
    }
(To see this, think of running the exploration process with stopping
time $T$ as in \refeq{Tdef}, where $\eta$ is defined by $2N=\lceil
\eta V\rceil$ (so in particular $\eta \ll \vep$).
Since $|C({\bf v})|\in (N,2N]$, we have $C({\bf v})=C_{\sss T}({\bf v})$.)
Letting
    \eqn{
    \lbeq{cnrel}
    \chs{\tcn}=\cn -2C \Big [\log{n}\vee \frac{N}{n} \Big ],
    }
we can lower bound
    \eqn{
    \Pr (|C({\bf v}_0)| \in (N,2N])
    \geq \prob_{\chs{\tcn}, p}(F\geq N)-\prob_{\cn, p}(F\geq 2N)+O(V^{-3}).
    }
Further, we can upper bound
    \eqan{
    &\Pr (|C({\bf v}_1)|\in (N,2N], {\bf v}_1 \not \in  C({\bf v}_0)\big||C({\bf v}_0)|
    \in (N,2N])\nonumber\\
    &\qquad
    =\Pr(|C({\bf v}_1)|\geq N, {\bf v}_1 \not \in  C({\bf v}_0)\big||C({\bf v}_0)| \in
    (N,2N])\nonumber\\
    &\qquad \quad
    -\Pr(|C({\bf v}_1)|\geq 2N, {\bf v}_1 \not \in  C({\bf v}_0)\big||C({\bf v}_0)| \in (N,2N])\nonumber\\
    &\qquad \leq \prob_{\cn, p}(F\geq N)-\prob_{\chs{\tcn}, p}(F\geq 2N)+O(V^{-3}).
    }
(Once again, to see this, think of first
exploring the cluster of ${\bf v}_0$ and, after that, the cluster of
${\bf v}_1$ in $H(2,n)$ with the cluster of ${\bf v}_0$ removed.)

Since $N\leq V^{2/3}$ and \ch{$N\ll\vep V$,} we can use
Lemma~\ref{lem.tail} to bound $\Pr \big(|C({\bf v}_0)| \in (N,2N]\big)$
and obtain
    \eqn{
    p_{{\bf v}_0, {\bf v}_1}
    \leq \frac{C}{\sqrt{N}}\Big(\prob_{\cn, p}(F\geq N)-\prob_{\chs{\tcn},
p}(F\geq N) +\prob_{\cn, p}(F\geq 2N)-\prob_{\chs{\tcn}, p}(F\geq 2N)\Big)+O(V^{-3}).
    }
By~(\ref{ingre1-diffBP}), with $\vep=p\cn-1$ and $\chs{\tvep}=p\chs{\tcn}-1$, for
every $\ell \in \bbN$,
    \eqn{
    \prob_{\cn, p}(F\geq \ell )-\prob_{\chs{\tcn}, p}(F\geq \ell)
    \leq C \Big(|\vep-\chs{\tvep}|+\frac{1}{n\ell^{1/2}}+\frac{1}{\ell^3}\Big).
    }
Note that, by \refeq{cnrel},
    \eqn{
    |\vep-\chs{\tvep}|= C\Big[\frac{\log{n}}{n}\bigvee \frac{N}{n^2}\Big],
    }
so that we always have $\frac{1}{n\ell^{1/2}}
=O(|\vep-\chs{\tvep}|)$. 

Consequently (with the value of $C$ adjusted between inequalities),
for all vertex pairs ${\bf v}_0, {\bf v}_1$,
    \eqn{
    p_{{\bf v}_0, {\bf v}_1}
    \leq \frac{C}{\sqrt{N}}\Big(|\vep-\chs{\tvep}|+\frac{1}{N^3}\Big).
    }
Summing over ${\bf v}_0, {\bf v}_1$,
    \eqn{
    S_{\sss {\bf v}_0\ncs {\bf v}_1}
    =\sum_{{\bf v}_0, {\bf v}_1} p_{{\bf v}_0, {\bf v}_1}
    \leq \frac{CV^2}{\sqrt{N}}\Big(|\vep-\chs{\tvep}|\vee \frac{1}{N^3} \Big)\ch{=\frac{C V^2}{\sqrt{N}}\Big(\frac{\log{n}}{n}\bigvee
    \frac{N}{V}\bigvee \frac{1}{N^3}\Big),}
    }
\ch{since $|\vep-\chs{\tvep}| =  O \Big ( \frac{\log{n}}{n}\vee
\frac{N}{n^2}\Big)$.}
\end{prf}

\subsection{Proof of Proposition~\ref{lem-varZ}}
\label{sec-PfvarZ} We are now ready to complete the proof of
Proposition~\ref{lem-varZ}. We will make essential use of Lemmas
\ref{prop-Zsmall} and \ref{prop-Zinbetween}. The \ch{choice $\vepnull$
in Proposition~\ref{lem-varZ} will be given by $\vepnull^{-2} =\NV$,
where $\NV$ is determined below.}

Let $\delta > 0$, and, for $i=0,1, \ldots, I-1$, let
    \eqn{
    \lbeq{deltaidef}
    \delta_i= \frac{\delta}{4\zeta(2)[(i+1)\wedge (I-i)]^2}.
    }
The reasons for our choice for $\{\delta_i\}_{i=0}^{I-1}$
will become apparent shortly. For now let us note that
    \eqn{
    \sum_{i=1}^{I-1} \delta_i \leq \delta/2.
    }

Recall the definition of $\bar{Z}_{\ge \ell}$ from \refeq{Zbar-def}
and the decomposition in \refeq{decomp-conc}.
We will prove that the right hand side of \refeq{decomp-conc}
is $o(1)$ for suitable $\Nul$ and $\NV$; the
conditions that $\Nul$ and $\NV$ must satisfy are as follows:
    \begin{align}
    \Nul&\ch{\ll n}, \qquad \Nul\ch{\ll \vep^2 V,}
    \qquad \Nul\gg \frac{(\log{n})^2}{n^2 \vep^4},
    \qquad \Nul\geq \Big(\frac{n}{\log{n}}\Big)^{1/3},
    \lbeq{N-ass}\\
    \NV&\ll \vep V, \qquad \NV\leq V^{2/3}.
    \lbeq{NV-ass}
    \end{align}
Finally, Proposition~\ref{lem-varZ} requires that $\NV \gg
\vep^{-2}$. \ch{As $\vep\gg
(\log{V})^{1/3}V^{-1/3}=(\log{V})^{1/3}n^{-2/3}$, the choices
$\Nul=n\vep^{1/2}(\log{V})^{1/2}$ and $\NV=V^{2/3}=n^{4/3}$ clearly
satisfy the bounds in \refeq{N-ass}--\refeq{NV-ass}; thus we have
proved that appropriate choices can be made.

Let us note that it is here that the condition $\vep\gg
(\log{n})^{1/3} V^{-1/3}$ in Theorem \ref{thm-maind2} arises. We
need to show concentration of measure for clusters of size $\Nul$,
which satisfies the constraint $\Nul \ll n$; for such clusters , we
are unable to control very precisely the number of vertices per
\chs{coordinate} line (see Proposition~\ref{prop-badlines}) -- this then
gives rise to the $\log n/n$ factor in Lemma~\ref{prop-Zinbetween},
and hence at this point in our proof.

We now prove that the concentration bound in Proposition
\ref{lem-varZ} holds.}
By \refeq{N-ass}, $\Nul$ satisfies the hypotheses of
Lemma~\ref{prop-Zsmall}; hence, using the Chebyshev inequality,
    \eqn{
    \prob_p\big(|\Zbar_{\sss \geq \ch{\Nul}}|\geq
    \delta \vep V/2\big)\leq \frac{4\var_p(Z_{\sss \geq \Nul})}{(\delta \vep V)^2}=o(1).
    }
\ch{Denote $N_i=2^{i+1}\Nul$, and recall the relation
between $\Nul$ and $\NV$ in \refeq{NV-choice}.}
Since $N_i\leq \NV$, \ch{\refeq{NV-ass} implies}
that \ch{$N_i\ll \vep V$} and
$N_i\leq V^{2/3}$ for each $i$. Therefore, applying
Lemma~\ref{prop-Zinbetween} to $N_i=2^{i+1}\Nul$ and using the
Chebyshev inequality, we obtain
    \eqan{
    \prob_p\Big(|\Zbar_{\sss \geq 2^{i+1}\Nul}-
    \Zbar_{\sss \geq 2^{i}\Nul}|\geq
    \delta_i \vep V\Big)
    &\leq \big(\delta_i \vep V)^{-2}
    \var_p\big(Z_{\sss \geq 2^{i+1}\Nul}-
    Z_{\sss \geq 2^{i}\Nul}\big)\nonumber\\
    &\leq \big(\delta_i \vep V)^{-2} \Big [
    \frac{C V^2}{\sqrt{N_i}}\Big(\frac{\log{n}}{n}\bigvee
    \frac{N_i}{V}\bigvee \frac{1}{N_i^3}\Big)\Big ].
    }

It follows that under our assumptions
    \eqn
    {\prob_p\big(|\Zbar_{\sss \geq \NV}|\geq
    \delta \vep V\big)
    \leq o(1)
    +\sum_{i=0}^{I-1} \frac{\frac{C V^2}{\sqrt{N_i}}\Big(\frac{\log{n}}{n}\bigvee
    \frac{N_i}{V}\bigvee \frac{1}{N_i^3}\Big)}{(\delta_i \vep V)^2}.
    }
Each term here is given by
    \eqn{
    \lbeq{intermbd}
    \frac{C V^2}{\sqrt{N_i}}\frac{\Big(\frac{\log{n}}{n}\bigvee
    \frac{N_i}{V}\bigvee \frac{1}{N_i^3}\Big)}{(\delta_i \vep V)^2}
    =\frac{C}{\sqrt{N_i}}\frac{\Big(\frac{\log{n}}{n}\bigvee
    \frac{N_i}{V}\bigvee \frac{1}{N_i^3}\Big)}{\delta_i^2 \vep^2}.
    }
By the last assumption in \refeq{N-ass}, for all $i$,
$\frac{\log{n}}{n}\geq \frac{1}{\Nul^3}\geq \frac{1}{N_i^3}$,
so that the last term is never equal to the maximum.
It follows that we need to upper bound
    \eqn
    {\prob_p\big(|\Zbar_{\sss \geq \NV}|\geq
    \delta \vep V\big)
    \leq o(1)
    +\sum_{i=0}^{I-1} \frac{
    \frac{C V^2}{\sqrt{N_i}}\Big(\frac{\log{n}}{n}\bigvee
    \frac{N_i}{V}\Big)}{(\delta_i \vep V)^2}.
    }
Letting $m$ be the smallest $i$ such that
    \eqn{
    \frac{\log{n}}{n}\leq \frac{N_i}{V},
    }
we can write
    \eqn{
    \lbeq{msplit}
    \sum_{i=0}^{I-1} \frac{
    \frac{C V^2}{\sqrt{N_i}}\Big(\frac{\log{n}}{n}\bigvee
    \frac{N_i}{V}\Big)}{(\delta_i \vep V)^2}
    =\sum_{i=0}^m \frac{C}{\sqrt{N_i}}\frac{\log{n}}{n\delta_i^2 \vep^2}
    +\sum_{i=m+1}^{I-1} C\frac{\sqrt{N_i}}{\delta_i^2 \vep^2V}.
    }
Using our definition of $\delta_i$ in~(\ref{deltaidef}), we can upper bound
$ \sum_{i=m+1}^{I-1}\frac{\sqrt{N_i}}{\delta_i^2}$ by
    \begin{align}
    \sum_{i=m+1}^{I-1}\frac{\sqrt{N_i}}{\delta_i^2}
    &\leq \frac{16\zeta(2)^2}{\delta^2} 2^{I/2}\sqrt{\Nul} \sum_{i=1}^{I-1} 2^{(i-I)/2}(I-i)^2
    \nonumber\\
    &\leq  \frac{16\zeta(2)^2}{\delta^2} 2^{I/2}\sqrt{\Nul} \sum_{k=1}^{\infty}
    k^2 2^{-k/2} \le C 2^{I/2}\sqrt{\Nul}=C\sqrt{\NV}.
    \end{align}
Hence the second sum in \refeq{msplit} is at most
    \eqn{
    \lbeq{bdsecterm}
    \sum_{i=m+1}^{I-1} C\frac{\sqrt{N_i}}{\delta_i^2 \vep^2V}
    \leq C\frac{\sqrt{\NV}}{\vep^2V}.
    }
We want the right hand side of~\refeq{bdsecterm} to be $o(1)$, which forces
    \eqn{
    \lbeq{restrNa}
    \NV=N_I\ll \vep^4V^2.
    }
The bound in \refeq{restrNa} holds, \ch{since $\NV=o(\vep V)$, by
the first constraint in \refeq{NV-ass}, and since $\vep^3 V\geq 1$.}

On the other hand, the first sum in \refeq{msplit} can be upper
bounded by
    \eqn{
    \lbeq{bdthirdterm}
    \sum_{i=0}^m \frac{C}{\sqrt{N_i}}\frac{\log{n}}{n\delta_i^2 \vep^2}
    \leq \frac{C \log{n}}{n\sqrt{\Nul}\vep^2\delta^2}=o(1),
    }
since $\Nul\gg \frac{(\log{n})^2}{n^2 \vep^4}$ by the third bound in
\refeq{N-ass}. This proves the required concentration bound,
\ch{thus establishing Proposition~\ref{lem-varZ} and Theorem
\ref{thm-maind2}.} \qed

\subsection*{Acknowledgement}
The work of RvdH was supported in part by Netherlands Organisation for
Scientific Research (NWO). The work of MJL was partly supported by the
Nuffield Foundation.

\newcommand\AAP{\emph{Adv. Appl. Probab.} }
\newcommand\JAP{\emph{J. Appl. Probab.} }
\newcommand\JAMS{\emph{J. \AMS} }
\newcommand\MAMS{\emph{Memoirs \AMS} }
\newcommand\PAMS{\emph{Proc. \AMS} }
\newcommand\TAMS{\emph{Trans. \AMS} }
\newcommand\AnnMS{\emph{Ann. Math. Statist.} }
\newcommand\AnnPr{\emph{Ann. Probab.} }
\newcommand\CPC{\emph{Combin. Probab. Comput.} }
\newcommand\JMAA{\emph{J. Math. Anal. Appl.} }
\newcommand\RSA{\emph{Random Struct. Alg.} }
\newcommand\SPA{\emph{Stoch. Proc. Appl.} }
\newcommand\ZW{\emph{Z. Wahrsch. Verw. Gebiete} }
\newcommand\PTRF{\emph{Probab. Theor. Relat. Fields}}
\newcommand\DMTCS{\jour{Discr. Math. Theor. Comput. Sci.} }

\newcommand\AMS{Amer. Math. Soc.}
\newcommand\Springer{Springer}
\newcommand\Wiley{Wiley}

\newcommand\vol{\textbf}
\newcommand\jour{\emph}
\newcommand\book{\emph}
\newcommand\inbook{\emph}
\def\no#1#2,{\unskip#2, no. #1,} 

\newcommand\webcite[1]{\hfil\penalty0\texttt{\def~{\~{}}#1}\hfill\hfill}

\def\nobibitem#1\par{}


\begin{thebibliography}{99}

\bibitem{AS00}
N.\ Alon and J.\ Spencer.
\newblock The Probabilistic Method, 2nd Edition.
\newblock John Wiley and Sons, New York (2000).

\bibitem{an72}
K.B. Athreya and P.E. Ney, Branching Processes, Springer, Berlin, 1972.

\bibitem{bhj}
A.D. Barbour, L. Holst, S. Janson, Poisson Approximation, OUP, Oxford, 1992.


\bibitem{bchss1}
C. Borgs, J.T. Chayes, R. van der Hofstad, G. Slade and J. Spencer,
Random subgraphs of finite graphs: I. The scaling window under the
triangle condition. \RSA \vol{27} (2005), 137--184.

\bibitem{bchss2}
C. Borgs, J.T. Chayes, R. van der Hofstad, G. Slade and J. Spencer,
Random subgraphs of finite graphs: II.The lace expansion and the
triangle condition. \AnnPr \vol{33} (2005), 1886--1944.


\bibitem{bchss3}
C. Borgs, J.T. Chayes, R. van der Hofstad, G. Slade and J. Spencer,
Random subgraphs of finite graphs: III. The phase transition on the
$n$-cube. \ch{{\it Combinatorica} \vol{26} (2006), 395--410.}

\bibitem{BCKS99}
C.~Borgs, J.~T.~Chayes, H.~Kesten and J.~Spencer.
\newblock Uniform boundedness of critical crossing probabilities implies hyperscaling.
\newblock {\em Random Struct.\ Alg.}, {\bf 15} (1999), 368--413.

\bibitem{BCKS01}
C.~Borgs, J.~T.~Chayes, H.~Kesten and J.~Spencer.
\newblock The birth of the infinite cluster: finite-size scaling in percolation.
\newblock {\em Commun.\ Math.\ Phys.}, {\bf 224} (2001), 153--204.


\bibitem{cs}
L.S.~Chandran and C.R.~Subramanian, A spectral lower bound for the
treewidth of a graph and its consequences, preprint, available at
{www.mpi-sb.mpg.de/$\chs{\sim}$sunil/applyspectree.ps}

\bibitem{devroye}
L. Devroye, Branching Processes and Their Applications in the Analysis
of Tree Structures and Tree Algorithms, in Probabilistic Methods for
Algorithmic Discrete Mathematics, ed. M.Habib, C. McDiarmid,
J. Ramirez-Alfonsin and B. Reed, 249--314, Springer-Verlag, Berlin, 1998.

\bibitem{dwass}
M. Dwass, The total progeny in a branching process, \JAP{} \vol{6}
(1969), 682--686.

\bibitem{HeHof05}
M.~Heydenreich and R.~van der Hofstad.
\newblock Random graph asymptotics on high-dimensional tori.
\newblock \ch{{\em Commun.\ Math.\ Phys.}, {\bf 270} (2007), 335--358.}

\bibitem{HS04a}
R.~van~der Hofstad and G.~Slade.
\newblock Expansion in $n^{-1}$ for percolation critical values on the $n$-cube
  and ${\mathbb Z}^n$: the first three terms.
\newblock {\em Combin.\ Probab.\ Comput.} \vol{15} (2006), 695--713.

\ch{\bibitem{HS04b}
R.~van~der Hofstad and G.~Slade.
\newblock Asymptotic expansions in {$n\sp {-1}$} for percolation critical values on
the {$n$}-cube and {$\Bbb Z\sp n$},
\newblock \RSA \vol{27} (2005), 331--357.}

\bibitem{cycle}
S. Janson, Cycles and unicyclic components in random graphs, \CPC{}
\vol{12} (2003), 27--52.

\bibitem{Jans02}
S.~Janson, On concentration of probability, Contemporary Combinatorics, ed. B. Bollob\'as,
Bolyai Soc.\ Math.\ Stud.\ \vol{10} (2002), J\'anos Bolyai Mathematical Society, Budapest, 289--301.


\bibitem{giant}
S. Janson, D.E. Knuth, T. {\L}uczak \& B. Pittel,
The birth of the giant component,
\RSA
\vol{3} (1993),
233--358.


\bibitem{JLR}
S. Janson, T. \L uczak \& A. Ruci\'nski,
\book{Random Graphs},
\Wiley, New York, 2000.

\bibitem{kolchin}
V.F. Kolchin, Moments of degeneration of a branching process and
height of a random tree, Mathematical Notes of the Academy of Sciences
of the USSR \vol{6} (1978), 954--961.

\bibitem{lmu03}
M.J. Luczak, C. McDiarmid and E. Upfal, On-line routing of random
calls in networks, \PTRF{} \vol{125} (2003), 457--482.


\bibitem{cmcd98}
C. McDiarmid, Concentration, in Probabilistic Methods for
Algorithmic Discrete Mathematics, ed. M.Habib, C. McDiarmid,
J. Ramirez-Alfonsin and B. Reed, 195--248, Springer-Verlag, Berlin, 1998.

\bibitem{n07}
A. Nachmias, Mean-field conditions for percolation on finite graphs,
preprint.

\bibitem{otter}
R. Otter, The multiplicative process, \AnnMS{} \vol{20} (1949), 206--224.

\end{thebibliography}
\end{document}